\documentclass[11pt]{amsart}

\usepackage[nomath,nothms,nomisc,nolayout,nocolor,nohyperref,abbrs,arxiv]{mymacros}

\usepackage{amsmath,amssymb,bbm,mathtools}
\usepackage{xspace}
\usepackage{wrapfig}
\usepackage{xfrac}
\usepackage{verbatim}
\usepackage{thmtools, thm-restate}
\usepackage{subcaption}
\usepackage{fullpage}

\usepackage{soul}
\usepackage{microtype}

\newcommand{\black}{\black}
\usepackage[disable]{todonotes}   % notes showed disable

\usepackage[font=small,labelfont=bf]{caption}
\declaretheorem[numberwithin=section]{theorem}
\declaretheorem[sibling=theorem]{lemma}
\declaretheorem[sibling=theorem]{corollary}
\declaretheorem[sibling=theorem]{conjecture}
\declaretheorem[sibling=theorem]{proposition}

\declaretheorem[sibling=theorem]{fact}
\declaretheorem[numbered=yes]{remark}

\numberwithin{equation}{section}
\numberwithin{theorem}{section}

\usepackage{float}
\usepackage{cleveref}
\usepackage{tikz}

\usepackage{booktabs}

\usetikzlibrary{calc,decorations.markings}
\usetikzlibrary{decorations.pathmorphing}
\usetikzlibrary{decorations.pathreplacing}
\usetikzlibrary{patterns}
\usetikzlibrary{arrows}

\newcommand{\cc}[1]{\mathcal{#1}}

\newcommand{\ob}[1]{\left(#1\right )} %% Open braces
\newcommand{\cb}[1]{\left[#1\right ]} %% Closed braces
\newcommand{\ab}[1]{\langle #1 \rangle} %% Pair of angled braces
\newcommand{\V}[1]{\boldsymbol{#1}}
\newcommand{\blank}[1]{}
\newcommand{\beq}{\begin{equation}}
\newcommand{\eeq}{\end{equation}}

\newcommand{\Hhoro}{\tilde H}

\title{Random spanning forests and hyperbolic symmetry}
\author{Roland Bauerschmidt}
\address{University of Cambridge,
    Statistical Laboratory, DPMMS.}
\email{rb812@cam.ac.uk}

\author{Nicholas Crawford}
\address{The Technion, Department of
    Mathematics}
\email{nickc@tx.technion.ac.il}

\author{Tyler Helmuth}
\address{Durham University, Department of Mathematical Sciences.}
\email{tyler.helmuth@durham.ac.uk}

\author{Andrew Swan}
\address{University of Cambridge, Statistical Laboratory, DPMMS.}
\email{acks2@cam.ac.uk}

\begin{document}

\maketitle

\begin{abstract}
  We study (unrooted) random forests on a graph where the probability of a forest
  is multiplicatively weighted by a parameter $\beta>0$ per
  edge. This is called the arboreal gas model, and the special case
  when $\beta=1$ is the uniform forest model. 
  The arboreal gas can
  equivalently be defined to be Bernoulli bond percolation with parameter
  $p=\beta/(1+\beta)$ conditioned to be acyclic,
  or as the limit $q\to 0$ with $p=\beta q$ of the random cluster model.
  It is known that on the complete graph $K_{N}$ with
  $\beta=\alpha/N$ there is a phase transition similar to that of the
  Erd\H{o}s--R\'enyi random graph: a giant tree percolates for
  $\alpha > 1$ and all trees have bounded size for $\alpha<1$.
  In contrast to this, by exploiting an exact relationship between the
  arboreal gas and a supersymmetric sigma model with hyperbolic target
  space, we show that the forest constraint is significant in two
  dimensions: trees do not percolate on $\Z^2$ for any finite $\beta>0$. 
  This result is a consequence of a Mermin--Wagner theorem associated
  to the hyperbolic symmetry of the sigma model.
  Our proof makes use of two main ingredients:
  techniques previously developed for hyperbolic sigma models related
  to linearly reinforced random walks and a version of the principle of
  dimensional reduction.
\end{abstract}

\section{The arboreal gas and uniform forest model}
\label{sec:model}

\subsection{Definition and main results}
\label{sec:intro-defn}

Let $\bbG = (\Lambda,E)$ be a finite (undirected) graph.  A
\emph{forest} is a subgraph $F=(\Lambda, E')$ that does not contain
any cycles. We write $\cc F$ for the set of all forests. 
For $\beta>0$ the \emph{arboreal gas} (or \emph{weighted
  uniform forest model}) is the measure on forests $F$ defined by
\begin{equation} \label{e:P-forest}
  \P_{\beta}[F] \bydef \frac{1}{Z_{\beta}} \beta^{|F|},
  \qquad
  Z_{\beta} \bydef \sum_{F\in\cc F} \beta^{|F|},
\end{equation}
where $|F|$ denotes the number of edges in $F$.
It is an elementary observation that
the arboreal gas with parameter $\beta$ is precisely Bernoulli bond
percolation with parameter $p_{\beta}=\beta/(1+\beta)$ conditioned to be
acyclic:
\begin{equation} \label{e:Pperc}
  \P_{p_{\beta}}^{\rm perc}\cb{F \mid \text{acyclic}} 
  \bydef \frac{p_{\beta}^{\abs{F}}(1-p_{\beta})^{\abs{E}-\abs{F}}}{\sum_{F}
    p_{\beta}^{\abs{F}}(1-p_{\beta})^{\abs{E}-\abs{F}}}
  = \frac{ \beta^{\abs{F}}}{\sum_{F}\beta^{\abs{F}}}
  = \P_\beta[F].
\end{equation}
The arboreal gas model is  also the limit, as  $q\to 0$ with $p=\beta q$, of
the $q$-state random cluster model, see~\cite{MR1757964}.
The particular case $\beta=1$ is the \emph{uniform forest model} mentioned
in, e.g.,~\cite{MR1757964,MR2724667,MR2060630,MR2243761}.  We
emphasize that the uniform forest model is \emph{not} the weak limit
of a uniformly chosen spanning tree; emphasis is needed since the
latter model is called the `uniform spanning forest' (USF) in the
probability literature. We will shortly see that the arboreal gas has
a richer phenomenology than the USF.  In fact, in finite
volume, the uniform spanning tree
is the $\beta\to \infty$ limit of the arboreal gas.

Given that the arboreal gas arises from bond percolation, it is
natural to ask about the percolative properties of the arboreal
gas.
It is straightforward to rule out the occurrence of
percolation for small values of $\beta$ via the following
proposition, see Appendix~\ref{sec:Perc}.

\begin{proposition}
  \label{prop:dom}
  On any finite graph, the arboreal gas with parameter $\beta$ is
  stochastically dominated by Bernoulli bond percolation with
  parameter $p_{\beta}$.
\end{proposition}
In particular, all subgraphs of $\Z^{d}$, all trees
have uniformly bounded expectation if
$p_{\beta}<p_{c}(d)$ where $p_{c}(d)$ is the critical parameter for
Bernoulli bond percolation on $\Z^{d}$. 

In the infinite-volume
limit, the arboreal gas is a singular conditioning of bond
percolation, and hence the existence of a percolation transition as
$\beta$ varies is non-obvious. However, on the complete graph it is
known that there is a phase transition, see
\cite{MR1167294,MR3845513,MR2567041}. To illustrate some of our methods we will
give a new proof of the existence of a transition.

\begin{proposition} \label{thm:MFT0}
  Let $\E_{N,\alpha}$ 
  denote the expectation of the arboreal
  gas on the complete graph $K_{N}$ with $\beta = \alpha /N$, and let
  $T_{0}$ be the tree containing a fixed vertex $0$.  Then
  \begin{equation}
    \label{eq:MFT0}
    \E_{N,\alpha} 
    \abs{T_{0}} = (1+o(1))
    \begin{cases}
      \frac{\alpha}{1-\alpha} & \alpha<1 \\
      c N^{1/3} & \alpha = 1 \\
      (\frac{\alpha-1}{\alpha})^2 N & \alpha>1.
    \end{cases}
  \end{equation}
  where $c=3^{2/3}\Gamma(4/3)/\Gamma(2/3)$ and $\Gamma$ denotes the Euler Gamma function.
\end{proposition}
Thus there is a transition for the arboreal gas exactly as for the
Erd\H{o}s--R\'{e}nyi random graph with edge probability $\alpha/N$. To
compare the arboreal gas directly with the Erd\H{o}s--R\'{e}nyi
graph, recall that Proposition~\ref{prop:dom} shows the arboreal gas is
stochastically dominated by the Erd\H{o}s--R\'{e}nyi graph with edge
probability
$p_{\beta} = \beta - \beta^{2}/(1+\beta)$. The fact that the
Erd\H{o}s--R\'{e}nyi graph asymptotically has all components trees in
the subcritical regime $\alpha<1$ makes the behaviour of the arboreal gas
when $\alpha<1$ unsurprising. On the other hand, the conditioning plays a role when
$\alpha>1$, as can be seen at the level of the expected tree size.
For the supercritical Erd\H{o}s--R\'{e}nyi graph the expected size is
$4(\alpha-1)^{2}N$ as $\alpha\downarrow 1$ ---  this follows from the fact
that the largest component for the Erd\H{o}s--R\'{e}nyi graph with
$\alpha>1$ has size $yN$ where $y$ solves $e^{-\alpha y}=1-y$, see,
e.g.,~\cite{MR3524748}. For further discussion, see
Section~\ref{sec:lit}.

On $\Z^2$, the singular conditioning that defines the arboreal gas has
a profound effect. In the next theorem statement and henceforth, for
finite subgraphs $\Lambda$ of $\Z^{2}$ we write $\P_{\Lambda,\beta}$
for the arboreal gas on $\Lambda$.

\begin{theorem}
  \label{thm:Z2}
  For all $\beta>0$ there is a universal constant $c_\beta>0$ such that the
  connection probabilities satisfy
  \begin{equation}
    \label{eq:Z2}
    \P_{\Lambda,\beta}[0\leftrightarrow j] \leq |j|^{-c_\beta} \quad \text{for $j\in \Lambda \subset \Z^2$,}
  \end{equation}
  for all $\Lambda\subset\Z^2$, where
  `$i\leftrightarrow j$' denotes the event that the vertices $i$ and
  $j$ are in the same tree.
\end{theorem}

This theorem, together with classical techniques from percolation
theory, imply the following corollary for the infinite volume limit, see Appendix~\ref{sec:Perc}.
\begin{corollary}
  \label{cor:Z2}
  Suppose $\P_{\beta}$ is a translation-invariant weak limit of
  $\P_{\Lambda_{n},\beta}$ for an increasing exhaustion of finite volumes
  $\Lambda_{n}\uparrow \Z^{2}$. Then all trees are finite
  $\P_{\beta}$-almost surely.
\end{corollary}

Thus on $\Z^{2}$
the behaviour of the arboreal gas is completely different from that of
Bernoulli percolation. The absence of a phase transition can be
non-rigorously predicted from the representation of the arboreal gas
as the $q\to 0$ limit (with $p=\beta q$ fixed) of the random cluster
model with $q>0$~\cite{PhysRevLett.98.030602}. We briefly describe how
this prediction can be made. The critical point of the random cluster
model for $q\geq 1$ on $\Z^{2}$ is known to be
$p_{c}(q)=\sqrt{q}/(1+\sqrt{q})$~\cite{MR2948685}. Conjecturally, this
formula holds for $q>0$. Thus $p_{c}(q)\sim \sqrt{q}$ as
$q\downarrow 0$, and by assuming continuity in $q$ one obtains
$\beta_{c}=\infty$ for the arboreal gas. This heuristic applies
also to the triangular and hexagonal lattices. Our proof is in fact
quite robust, and applies to much more general recurrent
two-dimensional graphs. We have focused on $\Z^{2}$ for the
sake of concreteness.

This absence of percolation is not believed
to persist in dimensions $d\geq 3$: we expect that there is a
percolative transition on $\Z^{d}$ with $d\geq 3$.
In the next section we will discuss the conjectural behaviour of the arboreal
gas on $\Z^{d}$ for all $d\geq 2$. Before this, we outline how we
obtain the above results.  Our starting point is an alternate
formulation of the arboreal gas. Namely, in
\cite{MR3622573,MR2437672,MR2110547} it was noticed that the arboreal
gas can be represented in terms of a model of fermions, and that this
fermionic model can be extended to a sigma model with values in the
superhemisphere. We also use this fermionic representation, but our
results rely in an essential way on the new observation that this model is
most naturally connected to a sigma model taking values in a
\emph{hyperbolic} superspace. Similar sigma models have recently
received a great deal of attention due to their relationship with
random band matrices and reinforced random
walks~\cite{MR2728731,MR4021254,MR3420510,1907.07949}. We will
discuss the connection between our techniques and these papers after
introducing the sigma models relevant to the present paper. 
A key step in our proof is the following integral formula
for connection probabilities
in the arboreal gas (see Corollary~\ref{cor:horo-pin2}
for a version with general edge weights):
\begin{equation} \label{e:intro-mf}
  \P_{\Lambda,\beta}[0\leftrightarrow j]
  = \frac{1}{Z_\beta} \int_{\R^{\Lambda}} e^{t_j} e^{-\sum_{i\sim j} \beta (\cosh(t_i-t_j)-1)} \pa{ e^{-2\sum_{i}t_i} \det(-\Delta_{\beta(t)})}^{3/2}
  \delta_0(dt_0) \prod_{i\neq 0} \frac{dt_i}{\sqrt{2\pi}}
\end{equation}
where $\Delta_{\beta(t)}$ is the graph Laplacian with edge weights
$\beta e^{t_i+t_j}$, understood as acting on $\Lambda\setminus
0$. This formula is a consequence of the hyperbolic sigma model
representation of the arboreal gas.

Surprisingly, if the exponent $3/2$ in \eqref{e:intro-mf} is replaced
by $1/2$, then the integrand on the right-hand side is the mixing
measure of the vertex-reinforced jump process found by Sabot and
Tarr\`es~\cite{MR3420510}.  The Sabot--Tarr\`es formula (along with a
closely related version for the edge-reinforced random walk) is known
as the \emph{magic formula}~\cite{MR3469136}. It seems even more magical
to us that the same formula, with only a change of exponent, describes
the arboreal gas. We will explain in Section~\ref{sec:repr} that there are
in fact three ingredients to this magic: a `non-linear' version of the
matrix-tree theorem, supersymmetric localisation, and horospherical
coordinates for (super-)hyperbolic space.

We remark that the whole family of sigma models taking values in
hyperbolic superspaces has interesting behaviour,
but for the present
paper we restrict our attention to those related to the arboreal gas.
A more general discussion of such models can be found in~\cite{1912.05817}
by the second author.

\subsection{Context and conjectured behaviour}
\label{sec:intro-conj}

Recall that `$i\leftrightarrow j$' denotes the event that the vertices $i$ and
$j$ are in the same tree. We also write $\P_{\beta}\cb{ij}$ for the
probability an edge $ij$ is in the forest.

The following conjecture asserts that the
arboreal gas has a phase transition in dimensions $d\geq 3$, just as in 
mean-field theory (Proposition~\ref{thm:MFT0}).  
Numerical
evidence for this transition can be found in~\cite{PhysRevLett.98.030602}.

\begin{conjecture}
  \label{con:transition}
  For $d\geq 3$ there exists $\beta_c > 0$ such that
  \begin{equation} \label{eq:3d}
    \lim_{n\to\infty } \lim_{\Lambda\uparrow \Z^{d}} \E_{\Lambda,\beta}\frac{|T_0\cap B_{n}|}{|B_{n}|}
    \begin{cases}
      = 0 & (\beta < \beta_c)\\
      > 0 & (\beta > \beta_c)
    \end{cases}
  \end{equation}
where $T_0$ is the tree containing $0$ and $B_{n}$ is the
  ball of radius $n$ centred at $0$. 
Moreover, when $\beta<\beta_c$ there is a universal constant $c_\beta>0$ such that
  \begin{equation}
    \P_{\Lambda,\beta}[i\leftrightarrow j] \leq Ce^{-c_\beta|i-j|}, \qquad (i,j\in\Z^d).
  \end{equation}
  When $\beta > \beta_c$ there is a universal constant
    $c_{\beta}'> 0$ such that 
  \begin{equation}
    \label{eq:c-sup}
    \lim_{\Lambda\uparrow \Z^{d}}\P_{\Lambda,\beta}[i\leftrightarrow j] \geq c_\beta'. 
  \end{equation}
\end{conjecture}

As indicated in the previous section, it is straightforward to prove
the first equality of~\eqref{eq:3d} when $\beta$ is sufficiently
small.  The existence of a transition, i.e., a percolating phase for
$\beta$ large, is open.  However, a promising approach to proving the
existence of a percolation transition when $d\geq 3$ and $\beta\gg 1$
is to adapt the methods of~\cite{MR2728731}; we are currently pursuing
this direction. 
Obviously, the existence of a sharp transition, i.e., a precise
$\beta_{c}$ separating the two behaviours in~\eqref{eq:3d} is also open.
The next conjecture distinguishes the supercritical behaviour of
the arboreal gas from that of percolation for which the
  (centered) connection probabilities have exponential decay.
  \begin{conjecture}
    \label{con:Goldstone}
    For $d\geq 3$, when $\beta>\beta_{c}$ 
    \begin{equation}
      \label{eq:Goldstone}
      \lim_{\Lambda\uparrow \Z^{d}}\P_{\Lambda,\beta}[i\leftrightarrow j] - c_{\beta} \approx |i-j|^{-(d-2)},
      \qquad \text{as $|i-j|\to\infty$,}
    \end{equation}
    where $c_{\beta}'$ is the optimal constant for
    which~\eqref{eq:c-sup} holds.
  \end{conjecture}

Assuming the existence of a phase transition, one can also ask about
the \emph{critical} behaviour of the arboreal gas.  One intriguing
aspect of this question is that the upper critical dimension is not
clear, even heuristically.  There is some evidence that the critical dimension of the
arboreal gas should be $d=6$, as for percolation, and opposed to $d=4$
for the Heisenberg model.  For further details, and for other
related conjectures, see~\cite[Section~12]{MR3622573}.  

Theorem~\ref{thm:Z2} shows that the behaviour of the arboreal
gas in two dimensions
is different from that of percolation. This difference would be
considerably strengthened by the following conjecture, which
first appeared in~\cite{MR2110547}.

\begin{conjecture}
  \label{con:massgap}
 For $\Lambda\subset \Z^{2}$, for any $\beta>0$ there exists a
   universal constant $c_\beta>0$ such that
  \begin{equation}
    \lim_{\Lambda\uparrow \Z^{2}}\P_{\Lambda, \beta}[i\leftrightarrow j] \approx e^{-c_\beta|i-j|}, \qquad (i,j\in\Z^2).
  \end{equation}
  As $\beta \to \infty$, the constant $c_{\beta}$
  is exponentially small in $\beta$:
  \begin{equation}
    c_\beta \approx e^{-c\beta}.
  \end{equation}
  In particular, $\E_{\beta}|T_0| \approx e^{c\beta} < \infty$ (with a
  different $c$) where $T_0$ is the tree containing $0$. 
\end{conjecture}

This conjecture is \emph{much} stronger than the main result of
the present paper, Theorem~\ref{thm:Z2}, which establishes only that
all trees are finite almost surely, a significantly weaker property
than having finite expectation.

Conjecture~\ref{con:massgap} is a version of the mass
gap conjecture for ultraviolet asymptotically free field theories. The
conjecture is based on the field theory representation discussed in
Section~\ref{sec:repr}, and supporting heuristics can be found in,
e.g., \cite{MR2110547}.  Other models with the same conjectural
feature include the two-dimensional Heisenberg model
\cite{Polyakov75},   the two-dimensional vertex-reinforced jump process
\cite{MR2728731} (and other $\HH^{n|2m}$ models with $2m-n\leq  0$, see~\cite{1912.05817}), 
the two-dimensional Anderson model
\cite{PhysRevLett.42.673}, and most prominently four-dimensional
Yang--Mills Theories \cite{Polyakov75,MR2238278}.

Let us briefly indicate discuss why Conjecture~\ref{con:massgap} seems
challenging. Note that in finite volume the (properly normalized)
arboreal gas converges weakly to the uniform spanning tree as
$1/\beta \to 0$, see Appendix~\ref{app:UST}.  For the uniform spanning
tree it is a triviality that $c_{\beta}=0$, and this is consistent
with the conjecture $c_\beta \approx e^{-c\beta}$ as $\beta\to\infty$.
On the other hand $c_\beta \approx e^{-c\beta}$ suggests a subtle
effect, not approachable via perturbative methods such as using
$1/\beta>0$ as a small parameter for a low-temperature expansion as
can be done for, e.g., the Ising model.  Indeed, since
$t \mapsto e^{-c/t}$ has an essential singularity at $t=0$, its
behaviour as $t=1/\beta\to 0$ cannot be detected at any finite order
in $t=1/\beta$.  The same difficulty applies to the other models
mentioned above for which analogous behaviour is conjectured.

The last conjecture we mention is the negative correlation
conjecture stated in \cite{MR1757964,MR2060630,MR2724667} and recently in
\cite{1806.02675,1902.03719}. This conjecture is also expected to hold
true for general (positive) edge weights, see Section~\ref{sec:repr-h02}.

\begin{conjecture}
  \label{con:NA}
  For any finite graph and any $\beta>0$ negative correlation
  holds: for distinct edges $ij$ and $kl$,
  \begin{equation} \label{e:NA}
    \P_\beta[ij,kl] 
    \leq \P_\beta[ij] 
    \P_\beta[kl]. 
  \end{equation}
  More generally, for all distinct edges $i_1j_1, \dots, i_nj_n$ and $m<n$,
  \begin{equation}
    \P_\beta[i_1j_1,\dots, i_nj_n] \leq \P_\beta[i_1j_1,\dots, i_mj_m] 
    \P_\beta[i_{m+1}j_{m+1},\dots, i_nj_n].
  \end{equation}
\end{conjecture}

The weaker inequality 
$\P_\beta[ij,kl] \leq 2\P_\beta[ij] \P_\beta[kl]$
was recently proved in \cite{1902.03719}. It is intriguing that the Lorentzian
signature plays an important role in both~\cite{1902.03719} and the
present work, but we are not aware of a direct relation.  An important
consequence of the full conjecture (with factor $1$) is the existence
of translation invariant arboreal gas measures on $\Z^d$; we prove
this in Appendix~\ref{sec:Perc}.

\begin{proposition} 
  \label{prop:subseq}
  Assume Conjecture~\ref{con:NA} is true. Suppose $\Lambda_{n}$ is an
  increasing family of subgraphs such that
  $\Lambda_{n}\uparrow \Z^{d}$, and let $\P_{\beta,n}$ be the arboreal
  gas on the finite graph $\Lambda_{n}$. Then the weak limit
  $\lim_{n}\P_{\beta,n}$ exists and is translation invariant. 
\end{proposition}

\begin{remark}
  The conjectured inequality \eqref{e:NA}
  can be recast as a reversed second Griffiths inequality. More
  precisely, \eqref{e:NA} can be rewritten in terms of the
  $\HH^{0|2}$ spin model introduced below in Section~\ref{sec:repr} as 
  \begin{equation}
    \avg{(u_i\cdot u_j)(u_k\cdot u_l)}_\beta -    \avg{u_i\cdot
      u_j}_\beta\,\avg{u_k\cdot u_l}_\beta \leq 0. 
  \end{equation}
  This equivalence follows immediately from the results in
  Section~\ref{sec:repr}.
\end{remark}

\subsection{Related literature}
\label{sec:lit}

The arboreal gas has received attention under various names. 
An important reference for our work is~\cite{MR2110547},
along with subsequent works by subsets of these authors and
collaborators~\cite{MR2437672,MR2450506,MR2567041,MR2529812,MR3622573,MR2142209}.
These authors considered the connection of the arboreal gas with the
\emph{antiferromagnetic} $\bbS^{0|2}$ model.

Our results are in part based on a re-interpretation of the
$\bbS^{0|2}$ formulation in terms of the hyperbolic $\HH^{0|2}$
model. At the level of infinitesimal symmetries these models are
equivalent. The power behind the hyperbolic language is that it allows
for a further reformulation in terms of the $\HH^{2|4}$ model, which
is analytically useful. The $\HH^{2|4}$ representation arises from a
dimensional reduction formula, which in turn is a consequence of
supersymmetric
localization~\cite{PhysRevLett.43.744,MR2031859,1812.04422}. Much of
Section~\ref{sec:repr} is devoted to explaining this. The upshot is
that this representation allows us to make use of techniques
originally developed for the non-linear $\HH^{2|2}$ sigma
model~\cite{MR863830,MR1134935,MR1411617,MR2728731,MR2736958} and the
vertex-reinforced jump process~\cite{MR3420510,MR3189433}.  In
particular, our proof of Theorem~\ref{thm:Z2} makes use of an
adaptation of a Mermin--Wagner argument for the $\HH^{2|2}$ model
\cite{1907.07949,1911.08579,MR4021254}; the particular argument
  we adapt is due to Sabot~\cite{1907.07949}.  For more on the connections between these
models, see~\cite{MR3420510,MR4021254}.

Conjecture~\ref{con:NA} seems to have first appeared in print
in~\cite{MR1774843}. Subsequent related works, including 
proofs for some special subclasses of graphs,
include~\cite{MR2060630,MR2410396,1902.03719,MR2836455}.

As mentioned before, considerably stronger results are known for the
arboreal gas on the complete graph. The first result in this direction
concerned forests with a fixed number of edges~\cite{MR1167294},
and later a fixed number of trees was considered~\cite{MR2567041}.
Later in~\cite{MR3845513} the arboreal
gas itself was considered, in the guise of the Erd\H{o}s--R\'{e}nyi
graph conditioned to be acyclic. In~\cite{MR1167294} it was understood
that the scaling window is of size $N^{-1/3}$, and results on the
behaviour of the ordered component sizes when
$\alpha = 1 +\lambda N^{-1/3}$ were obtained. In particular, the large
components in the scaling window are of size $N^{2/3}$. A very
complete description of the component sizes in the critical window was
obtained in~\cite{MR3845513}.

We remark on an interesting aspect of the arboreal gas that was first
observed in~\cite{MR1167294} and is consistent with
Conjecture~\ref{con:Goldstone}. Namely, in the supercritical regime,
the component sizes of the $k$ largest non-giant components are of
order $N^{2/3}$ \cite[Theorem~5.2]{MR1167294}. This is in contrast to
the Erd\H{o}s--R\'{e}nyi graph, where the non-giant components are of
logarithmic size. The critical size of the non-giant components is
reminiscent of self-organised criticality, see~\cite{MR2564286} for
example. A clearer understanding of the mechanism behind this
behaviour for the arboreal gas would be interesting.

\subsection{Outline}
\label{sec:outline}

In the next section we introduce the $\HH^{0|2}$ and $\HH^{2|4}$ sigma
models, relate them to the arboreal gas, and derive several useful
facts. In Section~\ref{sec:MFT} we use the $\HH^{0|2}$ representation
and Hubbard--Stratonovich type transformations to prove
Theorem~\ref{thm:MFT} by a stationary phase argument. In
Section~\ref{sec:Z2} we prove the quantitative part of
Theorem~\ref{thm:Z2}, i.e., \eqref{eq:Z2}. 
The deduction that all trees are finite almost surely follows from
adaptions of well-known arguments and is given in
Appendix~\ref{sec:Perc}. For the convenience of readers, we
briefly discuss the fermionic representation of \emph{rooted}
spanning forests and spanning trees in Appendix~\ref{app:UST}.

\section{Hyperbolic sigma model representation}
\label{sec:repr}

In \cite{MR2110547}, it was noticed that the arboreal gas has a
formulation in terms of fermionic variables, which in turn can be
related to a supersymmetric spin model with values in the
superhemisphere and negative (i.e., antiferromagnetic) spin couplings.
In Section~\ref{sec:repr-h02}, we reinterpret this fermionic model as
the $\HH^{0|2}$ model (defined there) with positive (i.e.,
ferromagnetic) spin couplings.  This reinterpretation has important
consequences: in Section~\ref{sec:repr-h24}, we relate the $\HH^{0|2}$
model to the $\HH^{2|4}$ model (defined there) by a form of
dimensional reduction applied to the target space.  Technically this
amounts to exploiting supersymmetric localisation associated to an
additional set of fields.  The $\HH^{2|4}$ model allows the
introduction of horospherical coordinates, which leads to an
analytically useful \emph{probabilistic} representation of the model 
as a gradient model with a non-local and non-convex
potential. This gradient model is very similar to gradient
models that arise in the study of linearly-reinforced random
walks. In fact, up to the power of a
determinant,
this representation is in terms of a measure that is identical to the
\emph{magic formula} describing the mixing measure of the
vertex-reinforced jump process, see \eqref{e:intro-mf}.

\subsection{$\HH^{0|2}$ model and arboreal gas}
\label{sec:repr-h02}

Let $\Lambda$ be a finite set, let
$\V{\beta} = (\beta_{ij})_{i,j\in\Lambda}$ be real-valued
symmetric edge weights, and let $\V{h} = (h_i)_{i\in\Lambda}$ be
real-valued vertex weights.
Throughout we will use this bold notation to
denote tuples indexed by vertices or edges.
For $f\colon \Lambda \to \R$,
we define the Laplacian associated with the edge weights by
\begin{equation}
  \Delta_{\beta}f(i) \bydef \sum_{j\in\Lambda}\beta_{ij}(f(j)-f(i)).
\end{equation}
The non-zero edge weights induce a graph $\bbG = (\Lambda, E)$,
i.e., $ij\in E$ if and only if $\beta_{ij}\neq 0$.

Let $\Omega^{2\Lambda}$ be a (real) Grassmann algebra (or exterior
algebra) with generators $(\xi_i,\eta_i)_{i\in\Lambda}$, i.e., all of
the $\xi_i$ and $\eta_i$ anticommute with each other.  For
$i,j\in\Lambda$, define the even elements
\begin{align}
  \label{eq:zH02def}
  z_i &\bydef \sqrt{1- 2\xi_i\eta_i} \bydef 1-\xi_i\eta_i
  \\
  \label{eq:H02IP}
  u_i \cdot u_j &\bydef -\xi_i\eta_j - \xi_j\eta_i - z_iz_j
                  =
                  -1 -\xi_i\eta_j - \xi_j\eta_i + \xi_i\eta_i + \xi_j\eta_j - \xi_i\eta_i\xi_j\eta_j.
\end{align}
Note that $u_i \cdot u_i = -1$ which we formally interpret as meaning
that $u_i=(\xi,\eta,z) \in \HH^{0|2}$ by analogy with the hyperboloid
model for hyperbolic space.  However, we emphasize that
`$\in \HH^{0|2}$' does not have any literal sense. Similarly we
write $\V{u} = (u_{i})_{i\in\Lambda}\in (\HH^{0|2})^{\Lambda}$.
The fermionic derivative $\partial_{\xi_i}$ is defined in the natural
way, i.e., as the odd derivation on that acts on $\Omega^{2\Lambda}$ by
\begin{equation}
  \partial_{\xi_i} (\xi_i F) \bydef F, \quad \partial_{\xi_i}F \bydef 0
\end{equation}
for any form $F$ that does not contain $\xi_{i}$. An analogous
definition applies to $\partial_{\eta_i}$.  The hyperbolic fermionic
integral is defined in terms of the fermionic derivative by
\begin{equation} 
\label{e:H02int}
  [F]_0
  \bydef \int_{(\HH^{0|2})^\Lambda} F
  \bydef \prod_{i \in \Lambda} \pa{\partial_{\eta_i} \partial_{\xi_i} \frac{1}{z_i}} F
  = \partial_{\eta_N}\partial_{\xi_N} \cdots \partial_{\eta_1}\partial_{\xi_1} \pa{\frac{1}{z_1\cdots z_N} F}
  \in \R
\end{equation}
if $\Lambda=\{1,\dots,N\}$. It is well-known that while the fermionic integral 
is formally equivalent to a fermionic derivative, it behaves in many
ways like an ordinary integral. The factors of $1/z$ make the
hyperbolic fermionic integral invariant under a fermionic
version of the Lorentz group; see \eqref{e:TH02inv-0}.

The $\HH^{0|2}$ sigma model action is the even form $H_{\beta,h}(\V{u})$ in
$\Omega^{2\Lambda}$ given by
\begin{equation}
  \label{eq:H}
  H_{\beta,h}(\V{u})
  \bydef \frac{1}{2}(\V{u},-\Delta_{\beta}\V{u}) + (\V{h},\V{z}-1)
  = \frac{1}{4} \sum_{i,j}\beta_{ij}(u_i-u_j)^2 + \sum_i h_i (z_i-1)
\end{equation}
where $(a,b) \bydef \sum_{i}a_{i}\cdot b_{i}$, with $a_i\cdot b_i$
interpreted as the $\HH^{0|2}$ inner product defined
by~\eqref{eq:H02IP}.
The corresponding unnormalised expectation $[\cdot]_{\beta,h}$
and normalised expectation $\avg{\cdot}_{\beta,h}$ are defined by
\begin{equation} \label{e:h02expectation}
  [F]_{\beta,h} \bydef [Fe^{-H_{\beta,h}}]_{0}, \qquad \avg{F}_{\beta,h}
  \bydef \frac{[F]_{\beta,h}}{[1]_{\beta,h}}, 
\end{equation}
the latter definition holding when $[1]_{\beta,h}\neq 0$.
In \eqref{e:h02expectation} 
the exponential of the even form $H_{\beta,h}$
is defined by the formal power series expansion, which truncates at
finite order since $\Lambda$ is finite.
For an introduction 
to Grassmann algebras and
integration as used in this paper, see~\cite[Appendix~A]{1904.01532}.

Note that the unnormalised expectation $[\cdot]_{\beta,h}$ is
well-defined for all real values of the $\beta_{ij}$ and $h_i$,
including negative values, and in particular $\V{h}=\V{0}$, $\V{\beta}=\V{0}$, or
both, are permitted. We will use the abbreviations $[\cdot]_\beta \equiv [\cdot]_{\beta,0}$ and $\avg{\cdot}_\beta\equiv \avg{\cdot}_{\beta,0}$.

The following theorem shows that the partition
function $[1]_{\beta,h}$ of the $\HH^{0|2}$ model 
is exactly the partition function of the arboreal gas $Z_\beta$
defined in \eqref{e:P-forest} when $\V{h}=\V{0}$, and that it is a
generalization the partition function when $\V{h} \neq \V{0}$ which we
will subsequently denote by $Z_{\beta,h}$.  This connection between
spanning forests and the antiferromagnetic $\bbS^{0|2}$ model, which
is equivalent to our ferromagnetic $\HH^{0|2}$ model, was previously
observed in \cite{MR2110547}. As mentioned earlier, our hyperbolic
interpretation will have important consequences in what follows.

\begin{theorem} \label{thm:H02forest}
  For any real-valued weights $\V{\beta}$ and $\V{h}$,
  \begin{equation}     \label{e:H02forest}
    [1]_{\beta,h}
    = \sum_{F\in\cc F} \prod_{ij\in F}\beta_{ij}\prod_{T\in F} (1+\sum_{i\in T}h_{i})
  \end{equation}
  where the inner product runs over the trees $T$ that make up the forest $F$.
\end{theorem}

For the reader's convenience and to keep our exposition self
contained, we provide a concise proof of Theorem~\ref{thm:H02forest}
below.  The interested reader may consult the original paper
\cite{MR2110547}, where they can also find generalizations to
hyperforests. The $\V{h}=\V{0}$ case of Theorem~\ref{thm:H02forest}
also implies the following useful representations of probabilities for
the arboreal gas.

\begin{corollary} \label{cor:spin} Let $\V{h}=\V{0}$ and
    assume the edge weights $\V{\beta}$ are non-negative. Then for all edges
  $ab$,
  \begin{equation}
    \label{e:edge-spin}
    \P_\beta\cb{ab}
    =
    \beta_{ab}\ab{ u_a\cdot u_b +1}_\beta,
  \end{equation}
  and more generally, for all sets of edges $S$,
  \begin{equation}
    \label{e:edges-spin}
    \P_\beta[S] = \ab{\prod_{ij \in S} \beta_{ij}
      (u_{i}\cdot u_{j}+1)}_\beta. 
  \end{equation}
  Moreover, for all vertices $a,b \in\Lambda$,
  \begin{equation}
    \label{e:conn-spin}
    \P_\beta[a\leftrightarrow b] = -\avg{z_az_b}_\beta =
    -\avg{u_a\cdot u_b}_\beta  = \avg{\xi_{a}\eta_{b}}_{\beta} = 1-    \ab{\eta_{a}\xi_{a}\eta_{b}\xi_{b}}_{\beta},
  \end{equation}
  and also 
  \begin{equation}
    \label{eq:expz0}
    \avg{z_a}_\beta=0.
  \end{equation}
\end{corollary}

We will prove Theorem~\ref{thm:H02forest} and Corollary~\ref{cor:spin}
in Section~\ref{S:H02-USF}, but first we establish some
integration identities associated with the symmetries of $\HH^{0|2}$.

\subsection{Ward Identities for $\HH^{0|2}$}
\label{sec:hypsym}

Define the operators
\begin{equation}
  \label{e:T}
  T \bydef \sum_{i\in \Lambda} T_i \bydef \sum_{i\in \Lambda} z_i\partial_{\xi_i}, 
  \qquad 
  \bar{T} \bydef \sum_{i\in \Lambda} \bar{T}_i \bydef \sum_{i\in \Lambda} z_i \partial_{\eta_i},
  \qquad S \bydef \sum_{i\in\Lambda} S_i \bydef \sum_{i\in\Lambda}
  (\eta_i\partial_{\xi_i} + \xi_i \partial_{\eta_i}).
\end{equation}
Using~\eqref{eq:zH02def}, one computes that these act on
coordinates as
\begin{alignat}{3}
  T\xi_a &= z_a, &\qquad T\eta_a &= 0, &\qquad Tz_a &= -\eta_a,\\
  \bar{T}\xi_a &= 0, &\qquad \bar{T}\eta_a &= z_a, &\qquad \bar{T}z_a &= \xi_a,\\
  S\xi_a &= \eta_a, &\qquad S\eta_a &= \xi_a, &\qquad S z_a &= 0.
\end{alignat}
The operator $S$ is an even derivation on $\Omega^{2\Lambda}$, meaning that it obeys the usual Leibniz rule $S(FG) = S(F)G + FS(G)$ for any forms $F,G$. 
On the other hand, the operators $T$ and $\bar T$ are odd derivations on
$\Omega^{2\Lambda}$, also called \emph{supersymmetries}.  This means
that if $F$ is an even or odd form, then $T(FG) = (TF)G \pm F(TG)$,
with `$+$' for $F$ even and `$-$' for $F$ odd. We remark that $T$ and $\bar T$
can be regarded as analogues of the infinitesimal Lorentz boost symmetries
of $\HH^{n}$, while $S$ is an infinitesimal symplectic symmetry.
In particular, the inner product \eqref{eq:H02IP} is
invariant with respect to these symmetries, in the sense that
\begin{equation} \label{e:uaubinv}
  T (u_a\cdot u_b)  = \bar T(u_a \cdot u_b) = S (u_a\cdot u_b) = 0.
\end{equation}
For $T$, this follows from
$T (u_a\cdot u_b) = T(-\xi_a\eta_b-\xi_b\eta_a-z_az_b) =
-z_a\eta_b-z_b\eta_a+\eta_az_b+\eta_bz_a=0$ since the $z_{i}$ are
even. Analogous computations apply to $\bar T$ and $S$.

A complete description of the infinitesimal symmetries of $\HH^{0|2}$ is given by the orthosymplectic Lie superalgebra $\mathfrak{osp}(1|2)$, which is spanned by the three operators described above, together with a further two symplectic symmetries; see \cite[Section~7]{MR2110547} for details.

\begin{lemma}
  \label{lem:hypsym}
  For any $a \in\Lambda$, the operators $T_a$, $\bar T_a$ and $S$ are
  symmetries of the non-interacting expectation
  $[\cdot]_0$ in the sense that, for any form $F$,
  \begin{equation} \label{e:TH02inv-0}
    [T_aF]_0 = [\bar T_a F]_0 = [S_aF]_0 = 0.
  \end{equation}
  Moreover, for any $\V{\beta}=(\beta_{ij})$ and $\V{h} = \V{0}$, also
  $T = \sum_{i\in\Lambda} T_i$ and $\bar T = \sum_{i\in\Lambda} \bar
  T_i$ are symmetries of the interacting expectation $[\cdot]_\beta$: 
  \begin{equation} 
    \label{e:TH02inv-beta}
    [TF]_\beta = [\bar T F]_\beta = 0,
  \end{equation}
  and similarly $S = \sum_{i\in\Lambda} S_i$ is a symmetry of $[\cdot]_{\beta,h}$
    for any $\V{\beta}$ and $\V{h}$.
\end{lemma}

\begin{proof}
  First assume that $\V{\beta}=\V{0}$. Then by \eqref{e:T},
  \begin{equation}
    [T_a F]_0
    = \int \prod_i \partial_{\eta_i} \partial_{\xi_i} \frac{1}{z_i} (T_aF)
    = \int \ob{\prod_{i\neq a} \partial_{\eta_i} \partial_{\xi_i} \frac{1}{z_i}} \partial_{\eta_a} \partial_{\xi_a} \partial_{\xi_a} F = 0
  \end{equation}
  since $(\partial_{\xi_a})^2$ acts as $0$ since any form can have at
  most one factor of $\xi_a$.  The same argument applies to $\bar T$,
  and a similar argument applies to $S$.
  
  We now show that this implies $T$ and $\bar T$ are also symmetries
  of $[\cdot]_\beta$.  Indeed, for any form $F$ that is even
  (respectively odd), the fact that $T$ is an odd derivation and the fact that
  $[\cdot]_0$ is invariant implies the integration by parts formula
  \begin{equation}
    [TF]_\beta = \pm [F(TH_{\beta})]_\beta, \qquad H_{\beta} =
    H_{\beta,0} = \frac14 \sum_{i,j\in\Lambda} \beta_{ij} (u_i-u_j)^2.
  \end{equation}
  For any $\V{\beta}$ the right-hand side vanishes since $TH_{\beta} = 0$ by
  \eqref{e:uaubinv}. A similar argument applies for $\bar T$.  Since
  every form $F$ can be written as a sum of an even and an odd form,
  \eqref{e:TH02inv-beta} follows. 

  The argument for $S$ being a symmetry of $\cb{\cdot}_{\beta,h}$ is similar.
\end{proof}

To illustrate the use of these operators, we give a proof of 
the identities on the right-hand side of~\eqref{e:conn-spin} and a proof of \eqref{eq:expz0}.
Define
\begin{equation}
  \label{eq:lambda}
  \lambda_{ab} \bydef z_b\xi_a, \qquad
  \bar\lambda_{ab} \bydef z_b\eta_a,  
\end{equation}
and note $T\lambda_{ab} = \xi_a\eta_b + z_az_b $ and
$\bar T \bar\lambda_{ab}= \xi_{b}\eta_{a}+z_{a}z_{b}$. Hence
\begin{equation}
  \label{e:conn-alg}
  \avg{u_a\cdot u_b}_\beta = \avg{z_az_b - T \lambda_{ab}-\bar T
    \bar\lambda_{ab}}_\beta = \avg{z_az_b}_\beta, 
\end{equation}
where the final equality is by linearity and
Lemma~\ref{lem:hypsym}. In particular,
 $\avg{z_{a}^{2}}_{\beta}=-1$.
Reasoning similarly, we obtain
\begin{align} \label{e:z0}
  \avg{z_a}_\beta &= \avg{T\xi_a}_\beta =0, \\
  \label{e:zazb1}
  \avg{z_a z_b}_{\beta}
  &=  \avg{T\lambda_{ab}}_\beta -\avg{\xi_a \eta_b}_{\beta}
  = -\avg{\xi_a \eta_b}_{\beta},
\end{align}
which proves~\eqref{eq:expz0}, and implies
$\avg{\xi_{a}\eta_{a}}_{\beta}=1$. Since
$z_az_b = (1-\xi_a\eta_a)(1-\xi_b\eta_b) = 1 - \xi_a\eta_a -
\xi_b\eta_b + \xi_a\eta_a \xi_b\eta_b$ this also gives
\begin{equation} \label{e:zazb2}
  -\avg{z_a z_b}_\beta = 1-\avg{\xi_a\eta_a \xi_b\eta_b}_\beta.
\end{equation}
Finally, we note that the symplectic symmetry and  $S(\xi_a\xi_b) = \xi_a\eta_b - \xi_b\eta_a$ imply
\begin{equation} \label{e:xieta-symmetry}
  \avg{\xi_a\eta_b}_{\beta,h} = \avg{\xi_b\eta_a}_{\beta,h}.
\end{equation}

\subsection{Proofs of Theorem~\ref{thm:H02forest}  and Corollary~\ref{cor:spin}}
\label{S:H02-USF}

Our first lemma relies on the identities of the previous section.
\begin{lemma} \label{lem:tree}
  For any forest $F$,
  \begin{equation} \label{e:tree}
    \qa{ \prod_{ij \in F} (u_i\cdot u_j+1)}_0 = 1.
  \end{equation}
\end{lemma}

\begin{proof}
  By factorization for fermionic integrals, it suffices to prove
  \eqref{e:tree} when $F$ is in fact a tree.  We recall the definition
\begin{equation}
  [G]_0
  = \prod_i \partial_{\eta_i}\partial_{\xi_i} \frac{1}{z_i} G
= \prod_i \partial_{\eta_i}\partial_{\xi_i} (1+\xi_i\eta_i) G.
\end{equation}
Hence, if $T$ contains no
edges then we have $[1]_{0}=1$.
We complete the proof by induction, with the inductive assumption that
the claim holds for all trees on $k$ or fewer vertices.
To advance the induction, let $T$ be a tree on $k+1\geq 2$ vertices
and choose a leaf 
edge $\{a,b\}$ of $T$. We will advance the induction by considering
the sum of the integrals that result from expanding
$(u_{a}\cdot u_{b}+1)$ in~\eqref{e:tree}. 

Note that by Lemma~\ref{lem:hypsym}, if $G_{1}$ is even (resp.\  odd) and $TG=0$, then
\begin{equation}
  \label{eq:IBP}
  [(TG_{1})G]_0 = \mp [G_{1}(TG)]_0
\end{equation}
and similarly if $\bar T G = 0$. Thus for
such a $G$, recalling the definition~\eqref{eq:lambda} of
$\lambda_{ab}$ and $\bar\lambda_{ab}$, 
\begin{align}
  [(u_a\cdot u_b)G]_0
  = [(z_az_b-T\lambda_{ab}-\bar T\bar\lambda_{ab})G]_0
  = [z_az_bG]_0
  &= \frac12 [((T\xi_a)z_b+(\bar T\eta_a)z_b)G]_{0}
  \nnb
  &= \frac12 [(-\xi_a\eta_b+\eta_a\xi_b)G]_{0},
\end{align}
where we have used~\eqref{eq:IBP} in the second and final
equalities.  Applying this identity with
$G=\prod_{ij\in T\setminus\{a,b\}}(u_{i}\cdot u_{j}+1)$, the
right-hand side is $0$ since the product does not contain the missing
generator at $a$ to give a non-vanishing expectation. The inductive
assumption and factorization for fermionic integrals implies $[G]_0=1$, and thus
\begin{equation}
  [\prod_{ij \in T}(u_{i}\cdot u_{j}+1)]_0
  =
  [(u_a\cdot u_b+1)G]_0 = [G]_0 = 1,
\end{equation}
advancing the induction.
\end{proof}

\begin{lemma} 
  \label{lem:cycle} 
  For any $i,j\in\Lambda$ we have
  $(u_i\cdot u_j+1)^2=0$, and for any graph $C$ that contains a cycle,
  \begin{equation} \label{e:cycle}
    \prod_{ij \in C} (u_i\cdot u_j+1) = 0.
  \end{equation}
\end{lemma}

\begin{proof}
  It suffices to consider when $C$ is a cycle or doubled
  edge. Orienting $C$, the oriented edges of $C$ are
  $(1,2),\dots, (k-1,k),(k,1)$ for some $k\geq 2$. Then, with the
  convention $k+1=1$,
  \begin{align}
    \prod_{i=1}^{k} (u_i\cdot u_{i+1}+1)
    &= 
      \prod_{i=1}^{k}(-\xi_{i}\eta_{i+1}+\eta_{i}\xi_{i+1} + \xi_{i}\eta_{i}+\xi_{i+1}\eta_{i+1} -
      \xi_{i}\eta_{i}\xi_{i+1}\eta_{i+1}) \nnb
      &=\prod_{i=1}^{k}(-\xi_{i}\eta_{i+1}+\eta_{i}\xi_{i+1} +
        \xi_{i}\eta_{i}+\xi_{i+1}\eta_{i+1})
        ,
  \end{align}
  the second equality by nilpotency of the generators and $k\geq
  2$. To complete the proof of the claim we consider which terms are
  non-zero in the expansion of this product. First consider the term
  that arises when choosing $\xi_{1}\eta_{1}$ in the first term in the
  product: then for the second term any choice other than
  $\xi_{2}\eta_{2}$ results in zero. Continuing in this manner, the
  only non-zero contribution is
  $\prod_{i=1}^{k}\xi_{i}\eta_{i}$. Similar arguments apply to the
  other three choices possible in the first product, leading to 
  \begin{align}
    \prod_{i=1}^{k}(-\xi_{i}\eta_{i+1}+\eta_{i}\xi_{i+1} +
        \xi_{i}\eta_{i}+\xi_{i+1}\eta_{i+1}) &=
        \prod_{i=1}^{k}\xi_{i}\eta_{i} +
        \prod_{i=1}^{k}\xi_{i+1}\eta_{i+1} +
        \prod_{i=1}^{k}(-\xi_{i}\eta_{i+1}) +
        \prod_{i=1}^{k}\eta_{i}\xi_{i+1} \nnb
    &= (1 + (-1)^{k}+(-1)^{2k-1}+(-1)^{k-1}) \prod_{i=1}^{k}\xi_{i}\eta_{i}
  \end{align}
  which is zero for all $k$. The signs arise from re-ordering the
  generators. We have used that $C$ is a cycle for the third and
  fourth terms.
\end{proof}

\begin{proof}[Proof of Theorem~\ref{thm:H02forest} when $\V{h}=\V{0}$]
  By Lemma~\ref{lem:cycle}, 
  \begin{equation}
    \label{eq:Rep1}
    e^{\frac12 (\V{u},\Delta_\beta \V{u})}
    = \sum_{S} \prod_{ij \in S} \beta_{ij} (u_i\cdot u_j+1)
    = \sum_{F} \prod_{ij \in F} \beta_{ij} (u_i\cdot u_j+1),
  \end{equation}
  where the sum runs over sets $S$ of edges and that over $F$ is over forests. 
  By taking the
  unnormalised expectation $[\cdot]_0$ we conclude from Lemma~\ref{lem:tree} that
  \begin{equation} \label{e:Zbetapf}
    Z_{\beta,0}
    = [e^{\frac12 (\V{u},\Delta_\beta \V{u})}]_0
    = \sum_{F} \prod_{ij \in F} \beta_{ij}.\qedhere
  \end{equation}
\end{proof}

To establish the theorem for $\V{h} \neq \V{0}$ requires one further
preliminary, which uses the idea of \emph{pinning} the spin
$u_{0}$ at a chosen vertex $0 \in \Lambda$.
Informally, this means that $u_0$
always evaluates to $(\xi,\eta,z) = (0,0,1)$. Formally, this
means the following.
To compute the pinned expectation of a function $F$ of the
forms $(u_{i}\cdot u_{j})_{i,j\in\Lambda}$, we replace
$\Lambda$ by $\Lambda_{0} = \Lambda \setminus \{0\}$, set
\begin{equation} 
  \label{e:h-pinned-02}
  h_j = \beta_{0 j},
\end{equation}
in $H_{\beta}$, and replace all instances of 
$u_0 \cdot u_j$ by $-z_j$ in both $F$ and $e^{-H_{\beta}}$. The
pinned expectation of $F$ is the hyperbolic fermionic
integral~\eqref{e:H02int} of this form with respect to the generators
$(\xi_{i},\eta_{i})_{i\in\Lambda_{0}}$. 
We denote this expectation by 
\begin{equation}
  [\cdot]_\beta^0, \quad \avg{\cdot}_\beta^0.
\end{equation}
This procedure gives a way to identify any function of the
forms $(u_{i}\cdot u_{j})_{i,j\in\Lambda}$ with a function of the
forms $(u_{i}\cdot u_{j})_{i,j\in\Lambda_{0}}$ and
$(z_{i})_{i\in\Lambda_{0}}$. To minimize the notation, we will
implicitly identify $u_{0}\cdot u_{j}$ with $-z_{j}$ when taking
  pinned expectations of functions $F$ of the $(u_{i}\cdot u_{j})$.

The following proposition relates the pinned and unpinned models. 

\begin{proposition}
  \label{prop:pinexpH02}
  For any polynomial $F$ in $(u_i\cdot u_j)_{i,j\in \Lambda}$,
  \begin{equation}
    \label{eq:pinexpH02}
    [F]_\beta^0 = [(1-z_0)F]_\beta,\qquad
    \avg{F}_{\beta}^0
    = \avg{(1-z_0) F}_\beta.
  \end{equation}
\end{proposition}

\begin{proof}
  It suffices to prove the first equation
  of~\eqref{eq:pinexpH02}, as this implies
  $[1]_{\beta}^{0}=[1-z_{0}]_{\beta}=[1]_{\beta}$ since
  $[z_{0}]_{\beta}=0$ by \eqref{e:z0}. 
  
  Since $1-z_0 = \xi_0\eta_0$, for any form $F$ that contains a factor
  of $\xi_0$ or $\eta_0$, we have $(1-z_0)F=0$.  Thus the expectation
  $[(1-z_0)F]_\beta$ amounts to the expectation with respect to
  $[\cdot]_{0}$ of $F e^{-H_\beta}$ with all terms containing factors
  $\xi_0$ and $\eta_0$ removed.  The claim thus follows from by
  computing the right-hand side using the observations that (i)
  removing all terms with factors of $\xi_0$ and $\eta_0$
  from $u_0\cdot u_i$ yields $-z_i$, and (ii)
  $\partial_{\eta_{0}}\partial_{\xi_{0}}\xi_{0}\eta_{0}z_{0}^{-1}=1$.
\end{proof}

There is a correspondence between pinning and external fields. If one
first chooses $\Lambda$ and then pins at $0\in\Lambda$, the result is
that there is an external field $h_j$ for all
$j\in\Lambda\setminus 0$. One can also view this the other way around,
by beginning with $\Lambda$ and an external field $h_j$ for all
$j\in\Lambda$, and then realizing this as due to pinning at an
`external' vertex $\delta\notin\Lambda$. This idea shows that
Theorem~\ref{thm:H02forest} with $\V{h}\neq \V{0}$ follows from the case
$\V{h}=\V{0}$; for the reader who is not familiar with arguments of this
type, we provide the details below.

\begin{proof}[Proof of Theorem~\ref{thm:H02forest} when $\V{h}\neq \V{0}$]
  The partition function of the arboreal gas with $\V{h} \neq \V{0}$ can be
  interpreted as that of the arboreal gas with $\V{h} \equiv \V{0}$ on a graph
  $\tilde G$ augmented by an additional vertex $\delta$ and with
  weights $\tilde \beta$ given by $\tilde \beta_{ij} = \beta_{ij}$ for
  all $i,j\in G$ and
  $\tilde \beta_{i\delta} = \tilde \beta_{\delta i} = h_i$.  Each
  $F'\in \cc F(\tilde G)$ is a union of $F\in \cc F(G)$ with a
  collection of edges $\{i_{r}\delta\}_{r\in R}$ for some
  $R\subset V(G)$. Since $F'$ is a forest, $\abs{T\cap R}\leq 1$ for
  each tree $T$ in $F$.  Moreover, for any $F\in \cc F(G)$ and any
  $R\subset V(G)$ satisfying $\abs{V(T)\cap R}\leq 1$ for each $T$ in
  $F$, $F\cup \{i_{r}\delta\}_{r\in R}\in \cc F(\tilde G)$.  Thus
  \begin{equation}
    Z_{\tilde\beta,0}^{\tilde G}
    = \sum_{F'\in \cc F(G_{\delta})} \prod_{ij \in F'} \beta_{ij}
    = \sum_{F\in \cc F(G)}\prod_{ij \in F'} \beta_{ij} \prod_{T\in
      F}(1+\sum_{i\in T}h_{i})
    = Z_{\beta,h}^G.
  \end{equation}
  To conclude, note that
  $[(1-z_\delta)F]_{\tilde \beta} = [F]_{\tilde \beta}$
  for any function $F$ with $TF=0$; 
  this follows from $[z_aF] = [(T\xi_a)F] = -[\xi_a(TF)] = 0$. The
  conclusion now follows from Proposition~\ref{prop:pinexpH02}
  (where $\delta$ takes the role of $0$ in that proposition), which
  shows $[(1-z_\delta)F]_{\tilde\beta} = [F]_{\beta,h}$.
\end{proof}

\begin{proof}[Proof of Corollary~\ref{cor:spin}]
Since $\P_\beta\cb{ab} = \beta_{ab}\frac{d}{d\beta_{ab}} \log Z$, we have
\begin{equation}
  \label{eq:edge-spin-pf}
  \P_\beta\cb{ab} = -\frac{1}{2}\beta_{ab}\ab{ (u_{a}-u_{b})^{2}},
\end{equation}
and expanding the right-hand side
yields~\eqref{e:edge-spin}. Alternatively, multiplying ~\eqref{eq:Rep1} by
$\beta_{ij}(1+u_i\cdot
u_j)$, using Lemma~\ref{lem:cycle}, and then applying
Lemma~\ref{lem:tree} yields the result. Similar considerations yield
\eqref{e:edges-spin}, and also show that
\begin{equation}
  \label{e:cspre}
  \P_{\beta}[i\nleftrightarrow j] = \avg{1+u_{i}\cdot u_{j}}_{\beta}.
\end{equation}
Therefore  $\P_{\beta}[i\leftrightarrow j] = -\avg{u_{i}\cdot u_{j}}_{\beta}$.
Together with the identities \eqref{e:conn-alg}--\eqref{e:zazb2},
this proves \eqref{e:conn-spin}. We already established~\eqref{eq:expz0}
in Section~\ref{sec:hypsym}.
\end{proof}

\subsection{$\HH^{2|4}$ model and dimensional reduction}
\label{sec:repr-h24}

In this section we define the $\HH^{2|4}$ model, and show that for a
class of `supersymmetric observables' expectations with respect to the
$\HH^{2|4}$ model can be reduced to expectations with respect to the
$\HH^{0|2}$ model.  To study the arboreal gas we will use this
reduction in reverse: first we express arboreal gas quantities as
$\HH^{0|2}$ expectations, and in turn as $\HH^{2|4}$ expectations. The
utility of this rewriting will be explained in the next section, but in
short, $\HH^{2|4}$ expectations can be rewritten as ordinary
integrals, and this carries analytic advantages.

The $\HH^{2|4}$ model is a special case of the following more general
$\HH^{n|2m}$ model. These models originate with Zirnbauer's $\HH^{2|2}$
model~\cite{MR1411617,MR2728731}, but makes sense for all
$n ,m \in \N$.  For fixed $n$ and $m$ with $n+m>0$, the $\HH^{n|2m}$
model is defined as follows.

Let $\phi^1,\dots, \phi^n$ be $n$ real variables, and let
$\xi^1,\eta^1,\dots,\xi^m,\eta^m$ be $2m$ generators of a Grassmann
algebra (i.e., they anticommute pairwise and are nilpotent of order
$2$). Note that we are using superscripts to distinguish
variables. \emph{Forms}, sometimes called \emph{superfunctions}, are
elements of $\Omega^{2m}(\R^n)$, where $\Omega^{2m}(\R^{n})$ is the
Grassmann algebra generated by $(\xi^k,\eta^k)_{k=1}^m$ over
$C^\infty(\R^n)$. See \cite[Appendix~A]{1904.01532} for details.  We
define a distinguished even element $z$ of $\Omega^{2m}(\R^n)$ by
\begin{equation}
  \label{eq:zHnm}
  z \bydef \sqrt{1+\sum_{\ell=1}^n (\phi^\ell)^2 + \sum_{\ell=1}^m (-2\xi^\ell\eta^\ell)}
\end{equation}
and let $u = (\phi,\xi, \eta,z)$.   Given a finite set $\Lambda$, we write
$\V{u} = (u_i)_{i\in\Lambda}$, where 
$u_i=(\phi_i, \xi_i, \eta_i, z_i)$ with
$\phi_i\in \R^{n}$ and
$\xi_i=(\xi_{i}^{1},\dots,\xi_{i}^{m})$ and
$\eta_i=(\eta_{i}^{1},\dots,\eta_{i}^{m})$, each $\xi_{i}^{j}$ (resp.\
$\eta_{i}^{j}$) a generator of $\Omega^{2m\Lambda}(\R^{n\Lambda})$. 
We define the `inner product'
\begin{equation}
  \label{eq:IP}
  u_{i}\cdot u_{j}
  \bydef
  \sum_{\ell=1}^{n}\phi^{\ell}_{i}\phi^{\ell}_{j} +
  \sum_{\ell=1}^{m} (\eta_{i}^{\ell}\xi_{j}^{\ell}-\xi_{i}^{\ell}\eta_{j}^{\ell})
  -z_{i}z_{j}
.
\end{equation}
Note that these definitions imply $u_i\cdot u_i = -1$. If $m=0$,
the constraint $u_i\cdot u_i=-1$ defines the hyperboloid model for
hyperbolic space $\HH^n$, as in this case $u_i \cdot u_j$ reduces to the
Minkowski inner product on $\R^{n+1}$. For this reason
we write $u_i\in\HH^{n|2m}$ and $\V{u}\in (\HH^{n|2m})^\Lambda$ and
think of $\HH^{n|2m}$ as a hyperbolic supermanifold.
As we do not need to enter into the details of this mathematical
object, we shall not discuss it further (see \cite{MR1411617} for further details).
We remark, however, that the expression
$\sum_{\ell=1}^{m} (-\xi_i^\ell\eta_j^\ell +\eta_i^\ell\xi_j^\ell)$ is
the natural fermionic analogue of the Euclidean inner product
$\sum_{\ell=1}^{n}\phi_i^\ell \phi_j^\ell$ and motivates the 
supermanifold terminology.

The general class of models of
interest are defined analogously to the $\HH^{0|2}$ model by the
action
\begin{equation}
  \label{eq:H}
  H_{\beta,h}(\V{u}) \bydef \frac{1}{2}(\V{u},-\Delta_{\beta}\V{u}) + (\V{h},\V{z}-1),
\end{equation}
where we now require $\V{\beta} \geq 0$ and $\V{h}\geq
0$, i.e.,
$\V{\beta}=(\beta_{ij})_{i,j\in\Lambda}$ and
$\V{h}=(h_i)_{i\in\Lambda}$ satisfy $\beta_{ij} \geq 0$ and $h_i\geq
0$ for all $i,j\in\Lambda$. We have again used the notation $(a,b) =
\sum_{i\in\Lambda}a_{i}\cdot b_{i}$ but where
  $\cdot$ now refers to \eqref{eq:IP}.  For a form $F \in
\Omega^{2m\Lambda}(\HH^{n})$, the corresponding unnormalised
expectation is
\begin{equation}
  \label{eq:exp}
  \cb{F}^{\HH^{n|2m}} \bydef \int_{(\HH^{n|2m})^\Lambda} F e^{-H_{\beta,h}}
\end{equation}
where the superintegral of a form $G$ is
\begin{equation}
  \int_{(\HH^{n|2m})^{\Lambda}}G \bydef 
  \int_{\R^{n\Lambda}} \prod_{i\in\Lambda}\frac{d\phi^1_{i} \dots
  d\phi^n_{i}}{(2\pi)^{n/2}} \, \partial_{\eta^1_{i}}\partial_{\xi^1_{i}}
  \cdots \partial_{\eta^m_{i}}\partial_{\xi^m_{i}}
  \ob{\prod_{i\in\Lambda}\frac{1}{z_{i}}} G,
\end{equation}
where the $z_{i}$ are defined by~\eqref{eq:zHnm}.

Henceforth we will only consider the $\HH^{0|2}$ and $\HH^{2|4}$
models, and hence we will write $x_{i}=\phi_{i}^{1}$ and
$y_{i}=\phi^{2}_{i}$ for notational convenience.  We will also assume
$\V{\beta}\geq 0$ and $\V{h}\geq 0$ to ensure both models are
well-defined.

\subsubsection*{Dimensional reduction}

The following proposition shows that, due to an internal
supersymmetry, all observables $F$ that are functions of
$u_i\cdot u_j$ have the same expectations under the $\HH^{0|2}$ and
the $\HH^{2|4}$ expectation. Here $u_i\cdot u_j$ is defined as
in~\eqref{eq:H02IP} for $\HH^{0|2}$, respectively as in~\eqref{eq:IP} for $\HH^{2|4}$.
In this section and henceforth we work under the convention that 
$z_i = u_\delta \cdot u_i$ with $u_\delta = (0,\dots, 0,1)$, and that
$(u_i\cdot u_j)_{i,j}$ refers to the collection of forms indexed by
$i,j \in \tilde \Lambda \bydef \Lambda \cup \{\delta\}$.  In other
words, functions of $(u_i\cdot u_j)_{i,j}$ are also permitted to
depend on $(z_i)_{i}$.

\begin{proposition}
  \label{prop:H02-H24}
  For any $F\colon \R^{\tilde\Lambda\times\tilde\Lambda}\to \R$ smooth
  with enough decay that the integrals exist, 
  \begin{equation}
    \label{eq:H02-H24}
    \cb{ F((u_{i}\cdot u_{j})_{i,j})}_{\beta,h}^{\HH^{0|2}} 
    =     
    \cb{ F((u_{i}\cdot u_{j})_{i,j})}_{\beta,h}^{\HH^{2|4}}.
  \end{equation}
\end{proposition}

In view of this proposition we will subsequently drop the superscript
$\HH^{n|2m}$ for expectations of observables $F$ that are functions of
$(u_i\cdot u_j)_{i,j}$. That is, we will simply write $\cb{F}_{\beta,h}$
for 
\begin{equation}
  [F]_{\beta,h} 
  =[F]_{\beta,h}^{\HH^{0|2}}
  =[F]_{\beta,h}^{\HH^{2|4}}.
\end{equation}
We will similarly write
$\avg{F}_{\beta,h}=\avg{F}_{\beta,h}^{\HH^{0|2}}=\avg{F}_{\beta,h}^{\HH^{2|4}}$
whenever $\cb{1}_{\beta,h}^{\HH^{2|4}}$ positive and finite.

The proof of Proposition~\ref{prop:H02-H24} uses the following
fundamental \emph{localisation theorem}.  To state the theorem,
consider forms in $\Omega^{2N}(\R^{2N})$ and denote the even
generators of this algebra by $(x_i,y_i)$ and the odd generators by
$(\xi_i,\eta_i)$. Then we define
\begin{equation}
  Q \bydef \sum_{i=1}^{N} Q_i\,,
  \qquad 
  Q_i \bydef \xi_i \ddp{}{x_i} + \eta_i \ddp{}{y_i} - x_i\ddp{}{\eta_i} + y_i \ddp{}{\xi_i}.
\end{equation}

\begin{theorem} \label{thm:F0}
  Suppose $F \in\Omega^{2N}(\R^{2N})$ is integrable and satisfies $QF=0$. Then
  \begin{equation}
    \int_{\R^{2N}} \frac{dx\, dy\, \partial_{\eta}\, \partial_{\xi}}{2\pi}\, F = F_0(0)
  \end{equation}
  where the right-hand side is the degree-$0$ part of $F$ evaluated at $0$.
\end{theorem}
A proof of this theorem can be found, for example, in~\cite[Appendix~B]{1904.01532}.

\begin{proof}[Proof of Proposition~\ref{prop:H02-H24}]
  To distinguish $\HH^{0|2}$ and $\HH^{2|4}$ variables, we write the latter as $u_i'$, i.e., 
  \begin{align}
    u_i \cdot u_j &= -\xi_i^1\eta_j^1-\xi_j^1\eta_i^1 - z_i z_j
                    \\
    u_i' \cdot u_j' &= x_ix_j + y_iy_j  -\xi_i^1\eta_j^1-\xi_j^1\eta_i^1 -\xi_i^2\eta_j^2-\xi_j^2\eta_i^2 - z_i' z_j'.
  \end{align}
  We begin by considering the case $N=1$, i.e., a graph with a single
  vertex.   Since $e^{-H_{\beta,h}(\V{u})}$ is a function of $(u_i\cdot
  u_j)_{i,j}$, we will absorb the factor of  $e^{-H_{\beta,h}(\V{u})}$ into the observable $F$ to ease the
    notation. 
  The $\HH^{2|4}$ integral can be written as
  \begin{equation}
    \int_{\HH^{2|4}}F
    = \int_{\R^{2}} \frac{dx \, dy}{2\pi}
      \, \partial_{\eta^1}\partial_{\xi^1}
      \, \partial_{\eta^{2}}\partial_{\xi^{2}}     \frac{1}{z'} F
    = 
      \partial_{\eta^1}\partial_{\xi^1} 
\int_{\R^2} \frac{dx \, dy}{2\pi}
      \, \partial_{\eta^{2}}\partial_{\xi^{2}} \frac{1}{z'} F
  \end{equation}
  where
  \begin{equation}
    z' = \sqrt{1+x^2 +y^2 - 2\xi^1\eta^1-
        2\xi^{2}\eta^{2}}
  \end{equation}
  and $\int_{\R^{2}}dx \, dy
  \, \partial_{\eta^{2}}\partial_{\xi^{2}} \frac{1}{z'} F$ is the
  form in $(\xi^{1},\eta^{1})$ obtained by integrating the coefficient functions
  term-by-term.
  Applying the localisation theorem (Theorem~\ref{thm:F0}) to the variables 
  $(x,y,\xi^{2},\eta^{2})$  gives, after noting $z'$ localises to
  $z = \sqrt{1-2 \xi^{1}\eta^{1}}$, 
  \begin{align}
    \int_{\R^2} \frac{dx \, dy}{2\pi}
    \, \partial_{\eta^{2}}\partial_{\xi^{2}}
    \frac{1}{z'} F((u_i'\cdot
    u_j')) = \frac{1}{z}F((u_i\cdot u_j)_{i,j}). 
  \end{align}
  Therefore
  \begin{equation}
    \int_{\HH^{2|4}} F((u'_i\cdot u_j')_{i,j})
    = \int_{\HH^{0|2}} F((u_i\cdot u_j)_{i,j})
  \end{equation}
  which is the claim.
  The argument for the case of general $N$ is exactly analogous.
\end{proof}

\subsection{Horospherical coordinates}

Proposition~\ref{prop:H02-H24} showed that `supersymmetric
observables' have the same expectations in the $\HH^{0|2}$ and the
$\HH^{2|4}$ model. This is useful because the richer structure
of the $\HH^{2|4}$ model allows the introduction of
\emph{horospherical coordinates}, whose importance was recognised in
\cite{MR2104878,MR2728731}.  We will shortly define horospherical
coordinates, but before doing this we
state the result that we will deduce using them.

For the statement of the proposition, we require the following definitions.
Let $-\Delta_{\beta(t),h(t)}$ be the matrix with $(i,j)$th element
$\beta_{ij}e^{t_{i}+t_{j}}$ for $i\neq j$ and $i$th diagonal
element $-\sum_{j\in\Lambda}\beta_{ij}e^{t_{i}+t_{j}}-h_{i}e^{t_{i}}$. Let
\begin{align}
  \label{eq:horo-interact}
  \tilde H_{\beta,h}(t,s)
  &\bydef \sum_{ij}\beta_{ij}(\cosh(t_{i}-t_{j})+\frac12 e^{t_i+t_j}(s_i-s_j)^2-1)
    \nnb &\qquad\qquad + \sum_{i}h_{i}(\cosh(t_{i})+\frac12 e^{t_i}s_i-1) - 2\log \det (-\Delta_{\beta(t),h(t)}) + 3\sum_i t_i
  \\
  \tilde H_{\beta,h}(t)
  &\bydef \sum_{ij}\beta_{ij}(\cosh(t_{i}-t_{j})-1) + \sum_{i}h_{i}(\cosh(t_{i})-1) - \frac{3}{2}\log \det (-\Delta_{\beta(t),h(t)}) + 3 \sum_i t_i
\end{align}
where we abuse notation by using the symbol $\tilde H_{\beta,h}$ both
for the function $\tilde H_{\beta,h}(t,s)$ and
$\tilde H_{\beta,h}(t)$.  Below we will assume that $\V{\beta}$ is
irreducible, by which we mean that $\V{\beta}$ induces a connected graph.

\begin{proposition} 
  \label{prop:H24-horo} 
  Assume $\V{\beta} \geq 0$ and $\V{h} \geq 0$ with $\V{\beta}$
  irreducible and $h_i>0$ for at least one $i\in\Lambda$.  For all
  smooth functions $F\colon \R^{2\Lambda} \to \R$, respectively
  $F\colon \R^\Lambda \to \R$, such that the integrals on the left-
  and right-hand sides converge absolutely,
 \begin{align} 
    \label{e:H24-horo-ts}
    \cb{ F((x_i+z_i)_{i}, (y_i)_i)}_{\beta,h}^{\HH^{2|4}}
    &=
    \int_{\R^{2\Lambda}} F((e^{t_i})_{i}, (e^{t_i}s_i)_i) e^{-\tilde H_{\beta,h}(t,s)} \prod_{i} \frac{dt_i\, ds_i}{2\pi}
    \\
    \label{e:H24-horo-t}
    \cb{ F((x_{i}+z_i)_{i})}_{\beta,h}^{\HH^{2|4}}
    &=
    \int_{\R^{\Lambda}} F((e^{t_i})_{i}) e^{-\tilde H_{\beta,h}(t)} \prod_{i} \frac{dt_i}{\sqrt{2\pi}}.
  \end{align}
  In particular, the normalising constant
  $\cb{1}_{\beta,h}^{\HH^{2|4}}$ is the partition function
  $Z_{\beta,h}$ of the arboreal gas.
\end{proposition}

Abusing notation further, we will denote either of the expectations on
the right-hand sides of~\eqref{e:H24-horo-ts} and~\eqref{e:H24-horo-t}
by $\q{\cdot}_{\beta,h}$, and we will write
$\avg{\cdot}_{\beta,h}$ for the normalised versions. Before giving the
proof of the proposition, which is essentially standard, we collect
some resulting identities that will be used later.

\begin{corollary}
  \label{cor:et}
  For all $\V{\beta}$ and $\V{h}$ as in Proposition~\ref{prop:H24-horo},
  \begin{equation}
    \label{e:et}
    \avg{e^{t_i}} _{\beta,h} =     \avg{e^{2t_i}}_{\beta,h} =
    \avg{z_i}_{\beta,h}, \quad     \avg{e^{3t_i}}_{\beta,h}=1
  \end{equation}
  and
  \begin{equation}  \label{e:xieta}
    \avg{s_is_je^{t_i+t_j}}_{\beta,h}
    = \avg{\xi_i\eta_j}_{\beta,h},
  \end{equation}
  where the left-hand sides are evaluated as on the right-hand side of
  \eqref{e:H24-horo-ts}, and the right-hand sides are given by the
  $\HH^{0|2}$ expectation \eqref{e:h02expectation}.
\end{corollary}

\begin{proof}
  To lighten notation, we write $\avg{\cdot} \bydef \avg{\cdot}_{\beta,h}$.
  For the $\HH^{2|4}$ expectation \eqref{eq:exp}, we have
  % Note that
  $\avg{x_i^qz_i^p} = 0$ whenever $q>0$ is an odd integer by the symmetry
  $x\mapsto -x$ (recall that $x=\phi^1$). Also note that
  \begin{equation} \label{e:x2h24}
    \avg{x_i^2} = \avg{y_i^2} = \avg{\xi_i^1\eta_i^1} = \avg{\xi_i^2\eta_i^2},
  \end{equation}
  where we emphasize that the superscript of $x_i^2$ denotes the
  square and the superscript of $\xi_i^2$ denotes the second component.
  These identies follow from the $x\leftrightarrow y$
  and $\xi_{i}^{1}\eta_{i}^{1}\leftrightarrow \xi_{i}^{2}\eta_{i}^{2}$
  symmetries of the $\HH^{2|4}$ model and $\avg{x_i^2+y_i^2-2\xi_i^1\eta_i^1} =0$
  by supersymmetric localisation, i.e., Theorem~\ref{thm:F0}.
  Since
  \begin{align}
    \avg{z_i^2} &= 1- 2\avg{\xi_i\eta_i}
    && \text{in $\HH^{0|2}$},
    \\
    \avg{z_i^2} &= 1+\avg{x_i^2+ y_i^2-2\xi_i^1\eta_i^1-2\xi_i^2\xi_i^2}
                  = 1- 2\avg{\xi_i^2\eta_i^2}
    && \text{in $\HH^{2|4}$},
  \end{align}
  and since the left-hand sides are equal by Proposition~\ref{prop:H02-H24}, we further see that
  the $\HH^{2|4}$ expectation \eqref{e:x2h24} equals the $\HH^{0|2}$ expectation $\avg{\xi_i\eta_i}$.
  Similarly, $\ab{x_{i}^{2}z_{i}} = \ab{y_{i}^{2}z_{i}} =
    \ab{\xi_{i}^{1}\eta_{i}^{1}z_{i}}=
    \ab{\xi_{i}^{2}\eta_{i}^{2}z_{i}}$. By using the preceding
    equalities and by expanding 
    $\ab{(-1+z_{i}^{2})z_{i}}=\ab{(u_{i}\cdot u_{i}+z_{i}^{2})z_{i}}$
    in both $\HH^{0|2}$ and $\HH^{2|4}$, one obtains
   \begin{equation}
     -2\ab{x_{i}^{2}z_{i}} = -\ab{z_{i}} + \ab{z_{i}^{3}} = -2\ab{\xi_{i}\eta_{i}},
 \end{equation}
 where the first expectation is 
 with respect to $\HH^{2|4}$ and the others are with respect to $\HH^{0|2}$.
Using these identities and \eqref{e:H24-horo-ts}, we then find
  \begin{align}
    \label{e:et-pf}
    \avg{e^{t_i}} &= \avg{x_i+z_i} = \avg{z_i} 
    \\
    \label{e:e2t-pf}
    \avg{e^{2t_i}} &= \avg{(x_i + z_i)^2} = \avg{x_i^2} + \avg{z_i^2} = \avg{\xi_i\eta_i} + \avg{1 - 2\xi_i\eta_i} = \avg{1-\xi_i\eta_i} = \avg{z_i}
    \\
    \label{e:e3t-pf}
    \avg{e^{3t_i}} &= \avg{(x_i + z_i)^3} = \avg{3x_i^2z_i} +
                     \avg{z_i^3} = 3\avg{\xi_{i}\eta_{i}}
+ \avg{1 - 3\xi_i\eta_i}
                     = 1
                     .
  \end{align}

  The identity \eqref{e:xieta} follows analogously:
    \begin{equation}
      \avg{s_is_je^{t_i+t_j}}= \avg{y_iy_j} = \frac12 \avg{\xi_i\eta_j+\xi_j\eta_i} = \avg{\xi_i\eta_j}
    \end{equation}
    where we used the generalisation of \eqref{e:x2h24} for
    the mixed expectation $\avg{x_ix_j}$ and that
    $\avg{\xi_i\eta_j} = \avg{\xi_j\eta_i}$, see \eqref{e:xieta-symmetry}.
\end{proof}

To describe the proof of Proposition~\ref{prop:H24-horo}
we now define horospherical coordinates for $\HH^{2|4}$.
These are a change of generators from the variables
$(x,y,\xi^{\gamma},\eta^{\gamma})$ with $\gamma=1,2$ to
$(t,s,\psi^{\gamma},\bar\psi^{\gamma})$, where 
\begin{equation}
  \label{eq:horo}
  x = \sinh t - e^{t}(\frac{1}{2}s^{2} + \bar\psi^{1}\psi^{1} +
  \bar\psi^{2}\psi^{2}), \;\;
  y = e^{t}s,\;\;
  \eta^{i} = e^{t}\bar\psi^{i}, \;\;
  \xi^{i} = e^{t}\psi^{i}.
\end{equation}
We note that $\bar\psi_{i}$ is simply notation to indicate a generator
distinct from $\psi_{i}$, i.e., the bar does not denote complex
conjugation, which would not make sense. In these coordinates the
action is quadratic in
$s, \bar\psi^{1},\psi^{1},\bar\psi^{2},\psi^{2}$. This leads to a
proof of Proposition~\ref{prop:H24-horo} by explicitly integrating out
these variables when $t$ is fixed via the following standard lemma,
whose proof we omit.
\begin{lemma}
  \label{lem:FG}
  For any $N \times N$ matrix $A$,
  \begin{equation}
    \pa{\prod_i \partial_{\eta_i}\partial_{\xi_i}} e^{(\xi,A\eta)} = \det A,
  \end{equation}
  and, for a positive definite $N\times N$ matrix $A$,
  \begin{equation}
    \int_{\R^{N}} e^{-\frac12 (s,As)} \, \frac{ds}{\sqrt{2\pi}} = (\det A)^{ -1/2}.
  \end{equation}
\end{lemma}

\begin{proof}[Proof of Proposition~\ref{prop:H24-horo}]
  The first step is to compute the Berezinian for the horospherical
  change of coordinates. This can be done as in \cite[Appendix
  A]{MR4021254}.  There is an $e^{t}$ for the $s$-variables and an
  $e^{-t}$ for each fermionic variable, leading to a Berezinian
  $z e^{-3t}$, i.e.,
  \begin{equation}
    \cb{F}_{\beta,h}^{\HH^{2|4}} = \int
    \ob{\prod_{i}ds_{i} dt_{i}\partial_{\psi_{i}^{1}}\partial_{\bar\psi^{1}_{i}}
    \partial_{\psi^{2}_{i}}\partial_{\bar\psi^{2}_{i}}} 
  Fe^{-\bar H_{\beta,h}(s,t,\psi,\bar\psi)} \prod_{i}\frac{e^{-3t_{i}}}{2\pi}.
  \end{equation}
  where $\bar H_{\beta,h}(s,t,\psi,\bar\psi)$ is $H_{\beta,h}$
  expressed in horospherical coordinates.

  The second step is to apply Lemma~\ref{lem:FG} repeatedly.
  To prove \eqref{e:H24-horo-t}, we apply it twice,
  once for $(\bar\psi^{1},\psi^{1})$ and once for
  $(\bar\psi^{2},\psi^{2})$. The lemma applies since $F$ does
  not depend on $\psi^{1},\bar\psi^{1},\psi^{2},\bar\psi^{2}$ by assumption.
  To prove \eqref{e:H24-horo-t}, we apply it three times,
  once for $(\bar\psi^{1},\psi^{1})$, once for $(\bar\psi^{2},\psi^{2})$,
  and once for $s$.
  Each integral contributes a power of $\det (-\Delta_{\beta(t),h(t)})$,
  namely $-1/2$ for the Gaussian and $+1$ for each fermionic Gaussian.
  This explains the coefficient $2$ in \eqref{e:H24-horo-ts} and the
  coefficient $3/2=2-1/2$ in \eqref{e:H24-horo-t}. 

  The final claim follows as the conditions that $\V{\beta}$
  induces a connected graph and some $h_{i}>0$ implies
  $\cb{1}_{\beta,h}^{\HH^{2|4}}$ is finite. The claim thus follows from
  Theorems~\ref{prop:H02-H24} and~\ref{thm:H02forest}.
\end{proof}

\subsection{Pinned measure for the $\HH^{2|4}$ model}
\label{sec:pinH24}

This section introduces a pinned version of the $\HH^{2|4}$ model and
relates it to the pinned $\HH^{0|2}$ model that was introduced in
Section~\ref{sec:hypsym}.  For the $\HH^{2|4}$ pinning means $u_{0}$
always evaluates to $(x,y,\xi^{1},\eta^{1},\xi^{2},\eta^{2},z) = (0,0,0,0,0,0,1)$.
As before, we implement this by replacing $\Lambda$ by
$\Lambda_0 = \Lambda \setminus \{0\}$ and setting
\begin{equation} \label{e:h-pinned}
  h_j = \beta_{0j},
\end{equation}
and replacing $u_0 \cdot u_j$ by $-z_j$.
We denote the corresponding expectations by
\begin{equation} \label{e:exp-pinn-24}
  [\cdot]_\beta^0, \quad \avg{\cdot}_\beta^0.
\end{equation}

We can relate the pinned and unpinned measures exactly as for
the $\HH^{0|2}$ model.

\begin{proposition}
  \label{prop:pinexp}
For any polynomial $F$ in $(u_{i}\cdot u_{j})_{i,j\in\Lambda}$,
\begin{equation}
  \label{eq:pinexp1}
  [F]_\beta^0 = [(1-z_0)F]_\beta,\qquad
  \avg{F}_{\beta}^0
  = \avg{(1-z_0) F}_\beta.
\end{equation}
Moreover, $[1]^0_\beta = [1]_\beta$ and hence for any pairs of vertices $i_kj_k$,
\begin{equation} \label{e:spin0}
  \avg{\prod_{k} (u_{i_k}\cdot u_{j_k}+1)}^0_\beta
  = \avg{\prod_{k} (u_{i_k}\cdot u_{j_k}+1)}_\beta
  .
\end{equation}
\end{proposition}
\begin{proof}
  The first equality in~\eqref{eq:pinexp1} follows by reducing the
  $\HH^{2|4}$ expectation to a $\HH^{0|2}$ expectation by
  Proposition~\ref{prop:H02-H24} (recall the convention that $z_0 = u_\delta\cdot u_0$),
  then applying 
  Proposition~\ref{prop:pinexpH02} for the $\HH^{0|2}$ expectation, and finally applying
  Proposition~\ref{prop:H02-H24} again (in reverse).
  The second equality in \eqref{eq:pinexp1} then follows by normalising
  using that $[1]_\beta^0 = [1-z_0]_\beta = [1]_\beta$ (as  in Proposition~\ref{prop:pinexpH02}).
  The equalities \eqref{e:spin0} follow from $[1]_\beta^0 = [1]_\beta$
  by  differentiating with
  respect to the $\beta_{i_kj_k}$.
\end{proof}

The next corollary expresses the pinned model in horospherical coordinates.
For $i,j \in\Lambda$, set
\begin{equation} \label{e:betat}
  \beta_{ij}(t) \bydef \beta_{ij} e^{t_i+t_j},
\end{equation}
and let $\tilde D_\beta(t)$ be the determinant of $-\Delta_{\beta(t)}$
restricted to $\Lambda_0 = \Lambda \setminus \{0\}$, i.e., the
determinant of submatrix of $-\Delta_{\beta(t)}$ indexed by
$\Lambda_{0}$. When $\V{\beta}$ induces a connected graph, this
determinant is non-zero, and
by the matrix-tree theorem it can be written as 
\begin{equation} \label{e:Dmatrixtree}
  \tilde D_\beta(t) = \sum_T \prod_{ij} \beta_{ij}e^{t_i+t_j}
\end{equation}
where the sum is over all spanning trees on $\Lambda$. For
$t\in \R^\Lambda$, then define
\begin{equation}
  \label{eq:H24-t-field}
  \tilde H_\beta^0(t) \bydef \frac12 \sum_{i,j} \beta_{ij}
  (\cosh(t_i-t_j)-1) - \frac{3}{2} \log \tilde D_\beta(t) - 3\sum_{i}
  t_i. 
\end{equation}

By combining Proposition~\ref{prop:pinexp} with
Proposition~\ref{prop:H24-horo}, we have the following representation
of the pinned measure in horospherical coordinates .

\begin{corollary}
  \label{cor:horo-pin1}
  For any smooth function $F\colon \R^\Lambda \to \R$ with sufficient decay,
  \begin{equation} \label{e:H24-horo-pinned}
    \q{F((x+z)_i)}_{\beta}^0
    =
    \int F((e^{t_i})_{i}) e^{-\tilde
      H_\beta^0(t)}\, \delta_0(dt_0) \prod_{i\neq 0} \frac{dt_i}{\sqrt{2\pi}}. 
  \end{equation}
\end{corollary}

\begin{proof}
  We recall the definition of the left-hand side, i.e., that the
  expectation $\q{\cdot}_\beta^0$ is defined in
  \eqref{e:h-pinned}--\eqref{e:exp-pinn-24} as the expectation
  on $\Lambda_{0}$ given by
  $\q{\cdot}^{0}_\beta = \q{\cdot}_{\tilde\beta,\tilde h}$ with
  $\tilde \beta_{ij}=\beta_{ij}$ and $\tilde h_i = \beta_{0i}$ for
  $i,j\in\Lambda_0$.  The equality now follows from
  \eqref{e:H24-horo-t}, together with the observation that
  $\Delta_{\beta(t)}|_{\Lambda_0}$ is
  $\Delta_{\tilde \beta(t),\tilde h(t)}$ if $t_{0}=0$.
\end{proof}

In view of \eqref{e:H24-horo-pinned} and since
  $[1]_{\beta}^{0}=Z_{\beta}$ by Proposition~\ref{prop:pinexp}, we again abuse notation somewhat
and write the normalised expectation of a function of $t=(t_i)_{i\in\Lambda}$ as
\begin{equation} \label{e:expectation-horo-pinned}
  \avg{F}_\beta^0 = \frac{1}{Z_{\beta}}\int_{\R^\Lambda} F((t_i)_i) e^{-\tilde H_\beta^0(t)} \delta_0(dt_0)\, \prod_{i \neq 0} \frac{dt_i}{\sqrt{2\pi}} .
\end{equation}

\begin{corollary}
  \label{cor:horo-pin2}
  The connection probabilities can be written as in terms of
  the pinned $\HH^{2|4}$ measure:
  \begin{equation} \label{e:connspin-horo}
    \P_\beta[0\leftrightarrow i] = \avg{e^{t_i}}_{\beta}^0.
  \end{equation}
  Moreover, for any vertex $i$,
    \begin{equation}
      \label{e:connspin-horo-pin}
      \avg{e^{3t_{i}}}_{\beta}^{0}=1.
    \end{equation}
\end{corollary}

\begin{proof}
  \eqref{e:connspin-horo} follows by applying first
  \eqref{e:conn-spin}, then \eqref{e:spin0}, then using the fact that
  $u_0\cdot u_i=-z_i$ under $\avg{\cdot}_\beta^0$, then using
  that $\avg{x_i}_\beta=0$ by symmetry, and finally applying
  \eqref{e:H24-horo-pinned}:
  \begin{equation}
    \P_\beta[0\leftrightarrow i] = -\avg{u_0\cdot u_i}_\beta
    = \avg{z_i}_\beta^0
    = \avg{z_i+x_i}_\beta^0
    = \avg{e^{t_i}}_\beta^0.
  \end{equation}
  The argument that $\avg{e^{3t_{i}}}_{\beta}^{0}=1$
  is identical to     \eqref{e:e3t-pf} with $\avg{\cdot}_\beta$ replaced by $\avg{\cdot}_\beta^0$.
\end{proof}

\section{Phase transition on the complete graph}
\label{sec:MFT}

The following theorem shows that on the complete graph the arboreal
gas undergoes a transition very similar to the percolation transition,
i.e., the Erd\H{o}s--R\'enyi graph. As mentioned in the
introduction, this result has been obtained
previously~\cite{MR1167294,MR3845513,MR2567041}. We have included
a proof only to illustrate the utility of the $\HH^{0|2}$ representation.
The study of spanning forests of the complete graph goes back to (at least)
R\'{e}nyi \cite{MR0115938} who obtained a formula which can be seen to
imply that their asymptotic number grows like $\sqrt{e}n^{n-2}$,
see~\cite{MR0274333}.  

Throughout this section we consider $\bbG = K_{N}$, the
complete graph on $N$ vertices
 with vertex set $\{0,1,2,\dots, N-1\}$, and we choose
$\beta_{ij} = \alpha/N$ with $\alpha>0$ fixed for all edges
$ij$. For notational simplicity we write $Z_{\beta}$ and
$\P_{\beta}$, i.e., we leave the dependence on $N$ implicit.

\begin{theorem} \label{thm:MFT}
  In the high temperature phase $\alpha < 1$,
  \begin{alignat}{2}
    \label{e:MFT-Z-asym-lt1}
    Z_\beta &\sim e^{(N+1)\alpha/2} \sqrt{1-\alpha},
    &
    \qquad
    \P_\beta[0\leftrightarrow 1] &\sim
    \qa{\frac{\alpha}{1-\alpha}}  \frac{1}{N}
                                   .
  \intertext{In the low temperature phase $\alpha > 1$,}
    \label{e:MFT-Z-asym-gt1}
    Z_\beta & \sim
              \frac{a^{N+3/2} e^{(a^2+N)/(2 a)}}{(a-1)^{5/2} N},
    &\qquad
      \P_\beta[0\leftrightarrow 1] &
                                     \sim \qa{\frac{\alpha-1}{\alpha}}^{2}
                                   .
    \intertext{In the critical case $\alpha = 1$,}
    \label{e:MFT-Z-asym-1}
    Z_\beta &\sim \frac{3^{1/6}\Gamma(\frac{2}{3})e^{(N+1)/2}}{N^{1/6}\sqrt{2\pi}},
      &\qquad
      \P_\beta[0\leftrightarrow 1] &\sim % \frac{3^{1/6}\Gamma(\frac{1}{3})^2}{2\pi N^{2/3}}                    .
      \qa{\frac{3^{2/3}\Gamma(\tfrac{4}{3})}{\Gamma(\tfrac{2}{3})}} \frac{1}{N^{2/3}}.
  \end{alignat}
\end{theorem}

\subsection{Integral representation}

The first step in the proof of the theorem is the following integral
representation that follows from a transformation of the fermionic
field theory representation from Section~\ref{sec:repr-h02}.
We introduce the effective potential
\begin{equation}
  V(\tilde z)
  \bydef 
  - P(i\alpha \tilde z),
  \qquad
  P(w) \bydef
  % \frac{w^2}{2\alpha} + w + \log(1 +h-w) =
  \frac{w^2}{2\alpha} + w + \log(1 -w)
\end{equation}
and set
\begin{equation} \label{e:MFT-F-F01}
  F(w)
  \bydef 
  1 - \frac{\alpha}{1-w},
  % 1 - \frac{\alpha}{1+h-w},
  \qquad
  F_{01}(w)
  \bydef
  % -\pa{\frac{h-w}{1+h-w}}^2
  % \pa{F(w) - \frac{2\alpha}{N(h-w)(1+h-w)}} = 
  -\pa{\frac{w}{1-w}}^2
  \pa{F(w) - \frac{2\alpha}{N(-w)(1-w)}}.
\end{equation}

\begin{proposition}
  For all $\alpha>0$ and all positive integers $N$,
  \begin{align}
    \label{e:MFT-Z-V}
    Z_{\beta} &= 
              e^{(N+1)\alpha/2} \sqrt{\frac{N\alpha}{2\pi}} \int_{\R} d\tilde z\, e^{-NV(\tilde z)} F(i\alpha\tilde z)
    \\
    \label{e:MFT-F01-V}
    Z_{\beta}[0\leftrightarrow 1]
            &=
              e^{(N+1)\alpha/2} \sqrt{\frac{N\alpha}{2\pi}} \int_{\R} d\tilde z\, e^{-NV(\tilde z)} F_{01}(i\alpha\tilde z),
  \end{align}
  where $Z_{\beta}[0\leftrightarrow 1] \bydef \P_{\beta}[0\leftrightarrow 1]Z_{\beta}$.
\end{proposition}

\begin{proof}
  We start from the representations of the partition functions in
  terms of the $\HH^{0|2}$ model, i.e., Theorem~\ref{thm:H02forest}
  and Corollary~\ref{cor:spin}, which we simplify using the assumption
  that the graph is the complete graph.  Let
  $(\Delta_\beta f)_i = \frac{\alpha}{N}\sum_{j=0}^{N-1} (f_i-f_j)$ be
  the mean-field Laplacian and $\V{h}= (h_i)_i$.  Then
\begin{align}
  \frac12 (\V{u},-\Delta_\beta \V{u})
  &= - (\V{\xi},-\Delta_\beta \V{\eta}) - \frac12 (\V{z},-\Delta_\beta \V{z})
  \nnb
  &= - (\V{\xi},-\Delta_\beta \V{\eta}) + \alpha \sum_{i=0}^{N-1} \xi_i\eta_i + \frac{\alpha}{2N} \pa{\sum_{i=0}^{N-1} z_i}^2 - \frac{\alpha N}{2}
  \\
  (\V{h}, \V{z}-\V{1}) &= - \sum_{i=0}^{N-1} h_i\xi_i\eta_i
                         .
\end{align}
In the sequel we will omit the range of sums and products when
there is no risk of ambiguity.

To decouple the two terms that are not diagonal sums we use the following
Hubbard--Stratonovich-type transforms 
in terms of auxiliary variables $\tilde\xi,\tilde\eta$ (fermionic) and
$\tilde z$ (real). Let $\mathbf{1}$ be the vector such that
$\mathbf{1}_i=1$ for all $0\leq i\leq N-1$.
\begin{align}
  \label{e:MFT-HS-xieta}
  e^{+(\V{\xi},-\Delta_\beta\V{\eta})}
  &=
    \frac{1}{N\alpha}
  \partial_{\tilde\eta} \partial_{\tilde\xi} e^{\alpha(\tilde \xi
    \mathbf{1} -\V{\xi},\tilde\eta \mathbf{1}-\V{\eta})} 
    =
        \frac{1}{N\alpha}
  \partial_{\tilde\eta} \partial_{\tilde\xi} \qa{
    e^{N\alpha\tilde\xi\tilde \eta} \prod_i e^{\alpha
    (\xi_i\eta_i-\tilde\xi\eta_i-\xi_i\tilde\eta)} } 
  \\
    \label{e:MFT-HS-z}
  e^{-\frac{\alpha}{2N} (\sum_i z_i)^2}
  &
  =
  \sqrt{\frac{N\alpha}{2\pi}} \int_\R d\tilde z \, e^{-\frac12 N\alpha \tilde z^2} e^{i\alpha \tilde z \sum_i z_i}.
\end{align}
The second formula is the formula for the Fourier transform of a Gaussian measure.
The first formula can be seen by making use of the following
identity. Write $Af \bydef \frac{1}{N}\sum_i f_i$
for the average of $f$, so that 
\begin{align} \label{e:xi-identity}
  \alpha(\tilde \xi \mathbf{1} -\V{\xi},\tilde\eta \mathbf{1}-\V{\eta})
  &=
    \alpha([\tilde \xi   -A\V{\xi}]\mathbf{1}-[\V{\xi}-(A\V{\xi}) \mathbf{1}],[\tilde\eta -A\V{\eta}]\mathbf{1}-[\V{\eta}-(A\V{\eta})\mathbf{1}])
    \nnb
  &=
    \alpha([\tilde \xi -A\V{\xi}]\mathbf{1},[\tilde\eta -A\V{\eta}]\mathbf{1})+\alpha(\V{\xi}-(A\V{\xi})\mathbf{1},\V{\eta}-(A\V{\eta})\mathbf{1})
    \nnb
  &=
    N\alpha(\tilde \xi -A\V{\xi})(\tilde\eta -A\V{\eta})+(\V{\xi},-\Delta_\beta\V{\eta}).
\end{align}
Using this identity the first equality in \eqref{e:MFT-HS-xieta} is
readily obtained by computing the fermionic derivatives, while the second equality follows by expanding the exponent. In 
the second line of \eqref{e:xi-identity} we used the orthogonality of constant functions with
the mean $0$ function $\V{\xi}-(A\V{\xi})\mathbf{1}$.  Finally, on the last line of
\eqref{e:xi-identity}, we used that $[\tilde \eta-A\V{\eta}]\mathbf{1}$ is a constant
to write the $\ell^2$ inner product as a product multiplied by a
factor $N$, and the factor $\alpha$ in the second term was absorbed into
$\Delta_{\beta}$.

Substituting \eqref{e:MFT-HS-xieta}--\eqref{e:MFT-HS-z} into \eqref{e:H02forest} gives
\begin{align}
  Z_{\beta,h}
  &= \prod_i \partial_{\eta_i} \partial_{\xi_i} \frac{1}{z_i} e^{-\frac12 (\V{u},-\Delta_\beta \V{u})-(\V{h},\V{z}-\V{1})}
    \nnb
  &=
  \frac{e^{N\alpha/2}}{\sqrt{2\pi N\alpha}}
  \int_{\R} d
  \tilde z \partial_{\tilde \eta} \partial_{\tilde \xi} \;
    e^{-\frac12 N\alpha \tilde z^2 + N\alpha \tilde\xi\tilde\eta+\alpha/2}
    \nnb
    &\qquad\qquad
    \prod_{i=1}^N \qa{\partial_{\eta_i}\partial_{\xi_i} \pa{
    \exp\pa{\alpha(\xi_i\eta_i-\tilde\xi\eta_i-\xi_i\tilde\eta)+i\alpha \tilde z (1-\xi_i\eta_i)-\alpha \xi_i\eta_i+(1+h_i)\xi_i\eta_i}}}
\end{align}
Simplifying the term inside the exponential gives
\begin{align}
  Z_{\beta,h}
&=
  \frac{e^{N\alpha/2}}{\sqrt{2\pi N\alpha}}
  \int_{\R} d
  \tilde z \partial_{\tilde \eta} \partial_{\tilde \xi} \;
    e^{-\frac12 N\alpha \tilde z^2 + N\alpha \tilde\xi\tilde\eta+N\alpha i\tilde z+\alpha/2}
    \nnb
    &\qquad\qquad
    \prod_{i=1}^N \qa{\partial_{\eta_i}\partial_{\xi_i} \pa{
    \exp\pa{(1+h_i-i\alpha\tilde z)(\xi_i\eta_i)-\alpha(\tilde\xi\eta_i+\xi_i\tilde\eta)}}}.
\end{align}
Since $(\tilde\xi\tilde\eta)^2=0$ and $(\tilde\xi\eta_i+\xi_i\tilde\eta)^3=0$,
the exponential can be replaced by its third-order Taylor expansion, giving
\begin{align}
  Z_{\beta,h}
  &=
  \frac{e^{(N+1)\alpha/2}}{\sqrt{2\pi N\alpha}}
  \int_{\R} d
  \tilde z \partial_{\tilde \eta} \partial_{\tilde \xi} \;
    e^{-N\alpha [\frac12\tilde z^2 - \tilde\xi\tilde\eta- i\tilde z]}
    \prod_{i} \qa{
    (1+h_i-i\alpha\tilde z) - \alpha^2\tilde\xi\tilde\eta}.
    \nnb
    &=
  \frac{e^{(N+1)\alpha/2}}{\sqrt{2\pi N\alpha}}
  \int_{\R} d
  \tilde z \partial_{\tilde \eta} \partial_{\tilde \xi} \;
    e^{-N\alpha [\frac12\tilde z^2 - \tilde\xi\tilde\eta- i\tilde z]}
    \prod_{i}
    (1+h_i-i\alpha\tilde z) \prod_{i} [1 - \frac{\alpha^2}{1+h_i-i\alpha\tilde z} \tilde\xi\tilde\eta]
\end{align}
Using again nilpotency of $\tilde\xi\tilde\eta$ this may be rewritten as
\begin{align}
  Z_{\beta,h}
  &=
    \label{E:HS}
  \frac{e^{(N+1)\alpha/2}}{\sqrt{2\pi N\alpha}}
  \int_{\R} d
  \tilde z \partial_{\tilde \eta} \partial_{\tilde \xi} \;
    e^{-N\alpha [\frac12\tilde z^2- i\tilde z]}
    \prod_{i}
    (1+h_i-i\alpha\tilde z) \qa{1 + \pa{N\alpha- \sum_{i} \frac{\alpha^2}{1+h_i-i\alpha\tilde z}} \tilde\xi\tilde\eta}
    .
\end{align}
Evaluating the fermionic derivatives gives the identity
\begin{equation} \label{e:Z-generalh}
  Z_{\beta,h}=
  \frac{e^{(N+1)\alpha/2}\alpha N}{\sqrt{2\pi N\alpha}}
  \int_{\R} d
  \tilde z  \;
    e^{-N\alpha [\frac12\tilde z^2 - i\tilde z]}
    \prod_{i=1}^N
    (1+h_i-i\alpha\tilde z)\qa{1 - \frac{\alpha}{N}\sum_i (1+h_i-i\alpha\tilde z)^{-1}}.
\end{equation}

To show \eqref{e:MFT-Z-V}--\eqref{e:MFT-F01-V} we now take $\V{h}=0$.
By definition the last bracket in \eqref{e:Z-generalh} is then
$F(i\alpha\tilde z)$ and the remaining integrand defines
$e^{-NV(\tilde z)}$, proving \eqref{e:MFT-Z-V}.  For
\eqref{e:MFT-F01-V} we use that $z_i = e^{z_i-1}$, and hence that
$[z_0z_1]_{\beta} = Z_{\beta,-1_0-1_1}$. Therefore
\eqref{e:Z-generalh} implies
\begin{align}
  [z_0z_1]_{\beta} =
  \frac{e^{(N+1)\alpha/2}\alpha N}{\sqrt{2\pi N\alpha}}
  \int_{\R} d
  \tilde z  \;
  e^{-NV(\tilde z)} \pa{\frac{-i\alpha \tilde z}{1-i\alpha\tilde z}}^2
  \qa{F(i\alpha\tilde z) + \frac{2\alpha}{N} \qa{ \frac{1}{1-i\alpha\tilde z}-\frac{1}{-i\alpha\tilde z}}
  }.
\end{align}
By definition, the integrand equals $-F_{01}(i\alpha \tilde z)$, so
together with the relation $Z_\beta[0\leftrightarrow 1] = -[z_0z_1]_{\beta}$,
which holds by \eqref{e:conn-spin}, the claim \eqref{e:MFT-F01-V} follows.
\end{proof}

\subsection{Asymptotic analysis} 

To apply the method of stationary phase to evaluate the asymptotics of the integrals,
we need the stationary points of $V$,
and asymptotic expansions for $V$ and $F$.
The first two derivatives of $P$ are
\begin{equation} \label{e:P-derivatives}
  P'(w)
  = \frac{w}{\alpha} + 1 - \frac{1}{1-w}
  % = \frac{w}{\alpha} + 1 - \frac{1}{1+h-w}
    ,\qquad
  P''(w)
  % = \frac{1}{\alpha} - \frac{1}{(1+h-w)^2}
  = \frac{1}{\alpha} - \frac{1}{(1-w)^2}
  .
\end{equation}
The stationary points are those $w=i\alpha\tilde z$ such that
% satisfy 
$P'(w)=0$. 
% with the notation $w = i\alpha\tilde z$.  
This equation can be rewritten as % becomes
\begin{equation}
  % w^2  - w(1+h -\alpha)- \alpha h=0
  w^2  - w(1 -\alpha) =0
  ,
\end{equation}
which has solutions $w=0$ and $w=1-\alpha$.
We call a root $w_0$ \emph{stable} if $P''(w_0) >0$ and \emph{unstable} if $P''(w_0)<0$.
For $\alpha<1$ the root $0$ is stable whereas $1-\alpha$ is unstable;
for $\alpha>1$ the root $1-\alpha$ is stable whereas $0$ is unstable;
for $\alpha=0$ the two roots collide at $0$ and $P''(0)=0$.

For the asymptotic analysis, we start with the nondegenerate case
$\alpha \neq 1$. First
% Next we
observe that we can view the right-hand sides of \eqref{e:MFT-Z-V}--\eqref{e:MFT-F01-V}
as contour integrals and can, due to analyticity of the integrand and the decay of $e^{-N\alpha\tilde z^2/2}$ when $\re \tilde z$ is large,
shift this contour to the horizontal line $\R+iw$ for any $w\in \R$.
We will then apply Laplace's method in the version given by the next theorem,
which is a simplified formulation of \cite[Theorem~7, p.127]{MR0435697}.

\begin{theorem}
\label{thm:Laplace-interior}
Let $I$ be a horizontal line in $\C$. % (to simplify the statement).
Suppose that $V,G\colon U \to\R$ are analytic in a neighbourhood $U$ of the contour $I$, that $t_0 \in I$ is such that
$V'$ has a simple root at $t_0$, % that $V''(t_0)>0$,
and that $\re (V(t)-V(t_0))$
is positive and bounded away from $0$ for $t$ away from $t_0$.
Then
\begin{equation}
  \label{e:Laplace-thm}
  \int_I e^{-NV(t)} G(t) \, dt
  \sim
  2e^{-NV(t_0)}
  \sum_{s=0}^\infty
  \Gamma (s+1/2) \frac{b_{s}}{N^{s+1/2}} ,
\end{equation}
where the notation $\sim$ means that the right-hand side is an asymptotic expansion for the left-hand side,
and the coefficients are given by (with all functions evaluated at $t_0$):
\begin{equation}
  \label{e:b01}
  b_0  = \frac{G}{(2 V'')^{1/2}},\qquad
  b_1 =
        \left(
        2G'' - \frac{2V'''G'}{V''}
        +
        \left[
        \frac{5V'''^2}{6V''^2} -\frac{V''''}{2V''}
        \right] G
        \right)
        \frac{1}{(2V'')^{3/2}}
        ,
\end{equation}
and with $b_{s}$ as given in \cite{MR0435697} for $s \ge 2$.
(Also recall that $\Gamma(1/2) = \sqrt{\pi}$ and that $\Gamma(s+1)=s\Gamma(s)$.)
\end{theorem}

For $\alpha\neq 1$, denote by $w_0$ the unique stable root.
As discussed in the previous paragraph, we can shift the contour to the line $\R-i \frac{w_0}{\alpha}$,
and the previous theorem implies that
\begin{multline} \label{e:Laplace} 
  \sqrt{\frac{N\alpha}{2\pi}}
  \int_\R e^{-NV(\tilde z)} G(\tilde z) d\tilde z
  \\
  =
  \sqrt{\frac{1}{\alpha P''}}
  e^{NP}
  \qa{
    F - \frac{1}{4NP''}
    \left(
      2F'' - \frac{2P'''F'}{P''}
      +
    \left[
      \frac{5P'''^2}{6P''^2} -\frac{P''''}{2P''}
    \right] F
  \right)
  + O\left(\frac{1}{N^{2}}\right)
  },
\end{multline}
with all functions on the right-hand side are evaluated at $w_0$.
From this the proof of Theorem~\ref{thm:MFT} for $\alpha\neq 1$ is an
elementary (albeit somewhat tedious) computation of the derivatives of
$P$ and $F$ and $F_{01}$ at $w_0$.

\begin{proof}[Proof of Theorem~\ref{thm:MFT}, $\alpha<1$]
  The stable root is $w_0=0$.
  By \eqref{e:Laplace} and elementary computations for the derivatives of $P$ and $F$ and $F_{01}$, we find
  \begin{align}
    \sqrt{\frac{N\alpha}{2\pi}} \int_\R e^{-NV(\tilde z)} F(i\alpha\tilde z) d\tilde z
    &\sim
      \sqrt{1-\alpha}
    \\
    \sqrt{\frac{N\alpha}{2\pi}}
    \int_\R e^{-NV(\tilde z)} F_{01}(i\alpha\tilde z) d\tilde z
    &\sim
      \frac{\alpha^2}{\sqrt{1-\alpha}}.
  \end{align}
  Recalling the definitions \eqref{e:MFT-Z-V}--\eqref{e:MFT-F01-V}, this implies the claims.
\end{proof}

\begin{proof}[Proof of Theorem~\ref{thm:MFT}, $\alpha>1$]
  The stable root is $w_0=1-\alpha$.
  Again \eqref{e:Laplace} and elementary computations for the derivatives of $P$ and $F$ and $F_{01}$
  lead to
  \begin{align}\label{e:MFT-gt1-Z}
    \sqrt{\frac{N\alpha}{2\pi}}
    \int_\R e^{-NV(\tilde z)} F(i\alpha\tilde z) d\tilde z
    &\sim
    e^{NP}    \frac{\alpha^{3/2}}{N(\alpha-1)^{5/2}}
    \\
        \sqrt{\frac{N\alpha}{2\pi}}
    \int_\R e^{-NV(\tilde z)} F_{01}(i\alpha\tilde z) d\tilde z
    &\sim
    e^{NP}    \frac{1}{N(\alpha-1)^{1/2}\alpha^{1/2}},
  \end{align}
  and $P = P(w_0)=P(1-\alpha)$. Again the claims follow from  \eqref{e:MFT-Z-V}--\eqref{e:MFT-F01-V}.
\end{proof}

At the critical point $\alpha=1$, the two roots collide at $0$ and $P''(0)=0$.
We analyse the integral as follows.

\begin{proof}[Proof of Theorem~\ref{thm:MFT}, $\alpha=1$]
We begin by using the conjugate flip symmetry to write
\begin{align}
  N^{\frac{2}{3}}\int_{\R} d\tilde z\, e^{-NV(\tilde z)} F(i\tilde z)
  = 2 N^{\frac{2}{3}}\re\int_{0}^\infty d\tilde z\, e^{-NV(\tilde z)} F(i\tilde z).
\end{align}
Using analyticity of the integrand, we then deform the contour from $[0, \infty)$ to $[0, e^{i\pi/6}\infty)$;
the contribution of the boundary arc
vanishes due to the decay of $e^{-N\alpha\tilde z^2/2}$ on this arc.
We now split the contour into two intervals $I_1 = [0, e^{i\pi/6}N^{-3/10})$ and $ I_2 = [e^{i\pi/6}N^{-3/10},e^{i\pi/6} \infty)$,
and denote the integrals over these regions as $J_1$ and $J_2$ respectively. 

Over the first interval $I_1$, we introduce the new real variable $s = N^{\frac{1}{3}}e^{-i\pi/6}\tilde{z}$, in terms of which
\begin{align}
  J_1 = 2N^{\frac{2}{3}}\re\int_{I_1} d\tilde z\, e^{-NV(\tilde z)} F(i\tilde z)
  &= 2\re\int_{0}^{N^{\frac{1}{30}}} ds\, e^{-NV(e^{\frac{i\pi}{6}}N^{-\frac{1}{3}}s)} N^{\frac{1}{3}}e^{\frac{i\pi}{6}}F(e^{\frac{2i\pi}{3}}N^{-\frac{1}{3}}s).
\end{align}
We then approximate the arguments as
\begin{align}
  NV(e^{\frac{i\pi}{6}}N^{-\frac{1}{3}}s) &= \frac{1}{3}s^3 + {O}(N^{-\frac{1}{3}}s^4)
  = \frac{1}{3}s^3 + {O}(N^{-\frac{6}{30}})\\
  N^{\frac{1}{3}}e^{\frac{i\pi}{6}}F(e^{\frac{2i\pi}{3}}N^{-\frac{1}{3}}s)
                                          &= e^{-\frac{i\pi}{6}}s + {O}(N^{-\frac{1}{3}}s^2)
                                          = e^{-\frac{i\pi}{6}}s + {O}(N^{-\frac{8}{30}}),
\end{align}
where the last error bounds hold uniformly for $s\in [0,N^{1/30}]$.
This gives
\begin{equation}
  J_1
  = 2\re\int_{0}^{N^\frac{1}{30}} ds\, e^{-\frac{i\pi}{6}}se^{-\frac{1}{3}s^3}  + {O}(N^{-\frac{4}{30}})
  = 2\re\int_{0}^{\infty} ds\, e^{-\frac{i\pi}{6}}se^{-\frac{1}{3}s^3}  + o(1)
  =3^{\frac{1}{6}}\Gamma(\tfrac{2}{3})+o(1).
\end{equation}

The second term $J_2$ is asymptotically negligible.
To see this, we bound $|F(i\tilde{z})| \le 1$, introduce the real
variable  $s = e^{-\frac{i\pi}{6}}\tilde{z}$, and split the resulting domain as $ [N^{-3/10},2) \cup [2,\infty) = I_2' \cup I_2''$:
\begin{equation}
J_2 = 2 N^{\frac{2}{3}}\re\int_{I_2} d\tilde z\, e^{-NV(\tilde z)} F(i\tilde z) \le   2 N^{\frac{2}{3}}\re\int_{I_2'} ds\, e^{-NV(\frac{i\pi}{6} s)} 
+2 N^{\frac{2}{3}}\re\int_{I_2''} ds\, e^{-NV(\frac{i\pi}{6} s)}
.
\end{equation}
Over $I_2'$, we use that $|I_2'| \le 2$ and bound the integral in
terms of the supremum of the integrand:
\begin{equation}
2 N^{\frac{2}{3}}\re\int_{I_2'} ds\, e^{-NV(\frac{i\pi}{6} s)} e^{\frac{i\pi}{6}}F(e^{\frac{2i\pi}{3}}s) \le 2 N^{\frac{2}{3}}\re\int_{I_2'} ds\, e^{-NV(\frac{i\pi}{6} s)}
\le 4 N^{\frac{2}{3}}\sup_{s \in I_2'} e^{-\re [NV(\frac{i\pi}{6} s)]},
\end{equation}
and as $\re NV(\frac{i\pi}{6} s)$ is decreasing, this supremum is attained on the boundary $s = N^{-3/10}$. Taylor expanding as before gives us 
\begin{equation}
4 N^{\frac{2}{3}}\sup_{s \in I_2'} e^{-\re NV(\frac{i\pi}{6} s)} = 4 N^{\frac{2}{3}}e^{-\re NV(\frac{i\pi}{6}N^{-\frac{3}{10}})}=e^{-(\frac{1}{3} + o(1))N^{\frac{1}{10}}}
\end{equation}
Over $I_2''$, we use that $\re [NV(\frac{i\pi}{6} s)] \ge \frac{Ns^2}{4}$ for all $s \ge 2$ to bound the second term as
\begin{equation}
2 N^{\frac{2}{3}}\re\int_{I_2'} ds\, e^{-NV(\frac{i\pi}{6} s)} \le 2 N^{\frac{2}{3}}\int_{I_2'} ds\, e^{-\frac{Ns^2}{4}} \le e^{-(1+o(1))N}.
\end{equation}

Putting together the estimates for $J_1$ and $J_2$, we therefore find
\begin{equation}
 N^{\frac{2}{3}}\int_{\R} d\tilde z\, e^{-NV(\tilde z)} F(i\tilde z) = J_1 + J_2 = 3^{\frac{1}{6}}\Gamma(\tfrac{2}{3}) + o(1)
\end{equation}
and hence the first asymptotic relation in \eqref{e:MFT-Z-asym-1}
follows from \eqref{e:MFT-Z-V}, i.e.,
\begin{equation}
Z_\beta \sim \frac{3^{\frac{1}{6}}\Gamma\left(\frac{2}{3}\right)e^{\frac{(N+1)}{2}}}{N^{\frac{1}{6}}\sqrt{2\pi}}.
\end{equation}

Using the same procedure, we can compute $\P_\beta[0\leftrightarrow 1]$.
We again split the (conveniently scaled) integral into two terms as
\begin{multline}
  N^{\frac{4}{3}}\int_{\R} d\tilde z\, e^{-NV(\tilde z)} F_{01}(i\tilde z)
  = 2 \re\int_{0}^{N^{\frac{1}{30}}} ds\, e^{-NV(e^{\frac{i\pi}{6}}N^{-\frac{1}{3}}s)} Ne^{\frac{i\pi}{6}}F_{01}(e^{\frac{2i\pi}{3}}N^{-\frac{1}{3}}s)
    \\
  + 2\re\int_{N^{\frac{1}{30}}}^{\infty} ds\, e^{-NV(e^{\frac{i\pi}{6}}N^{-\frac{1}{3}}s)} Ne^{\frac{i\pi}{6}}F_{01}(e^{\frac{2i\pi}{3}}N^{-\frac{1}{3}}s)
  = J_1 + J_2.
\end{multline}
As before $J_2$ is asymptotically negligible. For $J_1$, we approximate the $F_{01}$ term as
\begin{equation}
  Ne^{\frac{i\pi}{6}}F_{01}(e^{\frac{2i\pi}{3}}N^{-\frac{1}{3}}s)
  = e^{\frac{i\pi}{6}}s^3 + O(N^{-\frac{1}{3}}s^4)
  = e^{\frac{i\pi}{6}}s^3 + O(N^{-\frac{6}{30}}),
\end{equation}
uniformly for $s\in [0,N^{1/30}]$, to obtain the asymptotic relation
\begin{equation}
  J_1 = 2\re\int_{0}^{N^{\frac{1}{30}}} ds\, e^{-NV(e^{\frac{i\pi}{6}}N^{-\frac{1}{3}}s)} Ne^{\frac{i\pi}{6}}F_{01}(e^{\frac{2i\pi}{3}}N^{-\frac{1}{3}}s)
  \sim  2\re\int_{0}^{\infty} ds\, e^{\frac{i\pi}{6}}s^3e^{-\frac{1}{3}s^3}
  = 3^{\frac{5}{6}}\Gamma(\tfrac{4}{3}).
\end{equation}
From \eqref{e:MFT-F01-V}, we therefore find
\begin{equation}
Z_\beta[0\leftrightarrow 1] \sim \frac{3^{\frac{5}{6}}\Gamma\left(\frac{4}{3}\right)e^{\frac{(N+1)}{2}}}{N^{\frac{5}{6}}\sqrt{2\pi}}
\end{equation}
which after dividing by $Z_\beta$ shows the second asymptotic relation in \eqref{e:MFT-Z-asym-1}.
\end{proof}

\section{No percolation in two dimensions}
\label{sec:Z2}

In this section, we consider the arboreal gas on (finite approximations of) $\Z^2$ 
with constant nearest neighbour weights, i.e., with
$\beta_{ij}=\beta>0$ for all edges $ij$ and vertex weights $h_i=h$ for
all vertices $i$. As such we write $\beta$ instead of
  $\V{\beta}$ in this section. Constant weights are
merely a convenient choice; everything in this section also applies to
translation-invariant finite range weights, for example.  In contrast
with the case of the complete graph, we show that on $\Z^2$ the tree
containing a fixed vertex always has finite density.  Our arguments
are closely based on estimates developed for the vertex-reinforced
jump process \cite{MR4021254,1907.07949,1911.08579}.
The main new idea is to use these bounds in combination with dimensional reduction
from Section~\ref{sec:repr-h24}.

\subsection{Two-point function decay in two dimensions}

The proof of Theorem~\ref{thm:Z2}
makes use of the representation from
Section~\ref{sec:pinH24}, and closely follows~\cite{1907.07949};
an alternative proof could likely be obtained by adapting instead \cite{1911.08579}.

To lighten the notation, for a finite subgraph
$\Lambda\subset \Z^{2}$ we write $\P_{\beta}$ in place of
$\P_{\Lambda,\beta}$.
By \eqref{e:connspin-horo}, the connection
probability can be written in the horospherical coordinates of the
$\HH^{2|4}$ model as
\begin{equation} \label{e:connspin-horo-bis}
  \P_\beta[0 \leftrightarrow j] = \avg{e^{t_j}}_\beta^0
\end{equation}
where $\avg{\cdot}^0_\beta$ denotes the expectation with pinning at
vertex $0$.  Explicitly, by \eqref{e:expectation-horo-pinned}, the
measure $\avg{\cdot}_\beta^0$ on the right-hand side can be written as
the $a=3/2$ case of
\begin{equation}
  Q_{\beta,a}(dt) \bydef
\frac{1}{Z_{\beta,a}} \exp\ob{-\frac12 \sum_{i,j} \beta_{ij} (\cosh(t_i-t_j)-1)} D(\beta,t)^a \prod_{i\neq 0} \frac{dt_i}{ \sqrt{2\pi}},
\end{equation}
where
\begin{equation}
  D(\beta,t) \bydef
\tilde D_\beta(t) \prod_{i} e^{- 2 t_i},
\end{equation}
and where $\tilde D_{\beta}(t)$ was given explicitly in~\eqref{e:Dmatrixtree} and
$Z_{\beta,a}$ is a normalising constant. We have made the parameter $a$
explicit as our argument adapts that of~\cite{1907.07949}, which concerned the
case $a=1/2$. When $a=1/2$
supersymmetry implies that 
$Z_{\beta,1/2} =1$ and $\E_{Q_{\beta,1/2}}(e^{t_k})=1$ for all $\V{\beta}=(\beta_{ij})$ and all $k\in\Lambda$.
These identities require the following replacement when $a\neq 1/2$: 
  \begin{equation} \label{e:Zmon}
    Z_{\beta,a} \text{ is increasing in all of the $\beta_{ij}$},
    \qquad
    \E_{Q_{\beta,a}}(e^{2at_k}) = 1 \quad \text{for all $(\beta_{ij})$ and all $k\in\Lambda$.}
  \end{equation}
  When $a=3/2$ the first of these facts follow from the forest representation for
  the partition function, see Proposition~\ref{prop:H24-horo},
  and the second is \eqref{e:connspin-horo-pin} of Corollary~\ref{cor:horo-pin1}.
  Proof that \eqref{e:Zmon} holds for general half-integer $a \geq 0$
  appears in \cite{1912.05817}, and we conjecture that these
    assumptions are true for any $a\geq 0$.

With \eqref{e:Zmon} given, it is straightforward to adapt
\cite[Lemma 1]{1907.07949} to obtain the following lemma. In
the next lemma we assume $0,i\in\Lambda$, but we make no further
assumptions beyond that $\V{\beta}$ induces a connected graph.

\begin{lemma}[Sabot {\cite[Lemma 1]{1907.07949}} for $a=1/2$] \label{lem:Sabot}
  Let $a\geq 0$,  $s\in (0,1)$, and $\gamma >  0$. Assume~\eqref{e:Zmon} holds.
  Then for any $v \in \R^\Lambda$ with $v_j=1$, $v_{0}=0$, and
  \begin{equation}
    \gamma |v_i-v_k| \leq \frac12 (1-s)^2 \quad \text{for all $i\sim k$},
  \end{equation}
  one has, with $q=1/(1-s)$,
  \begin{equation} \label{e:Sabotbd}
    \E_{Q_{\beta,a}}(e^{2as t_j}) \leq e^{-2a s\gamma} e^{\frac12 \gamma^2q^2\sum_{i,k}(\beta_{ik}+2a)(v_i-v_k)^2}.
  \end{equation}
\end{lemma}
  
\begin{proof}
  As mentioned, our proof is an adaptation of
  \cite[Lemma~1]{1907.07949}, and hence we indicate the main steps
  but will be somewhat brief. In this reference
  $a=1/2$, $Q_{\beta,a}$ is denoted $Q$,
  $\beta_{ij}$ is denoted $W_{ij}$, and $t$ is denoted by $u$.
%  and we use the change of notation $W\to \beta$ and $u\to t$.
  Let $Q_{\beta,a}^\gamma$ denote the distribution of $t-\gamma
  v$. Since the partition function does not change under translation
  of the underlying measure, by following %and otherwise exactly as in
  \cite[Prop.~1]{1907.07949} we obtain,
  \begin{equation}
    \frac{dQ_{\beta,a}}{dQ^\gamma_{\beta,a}}(t)
    = \exp\pa{-\frac12 \sum_{i,k} \beta_{ik} (\cosh(t_i-t_k)-\cosh(t_i-t_k+\gamma(v_i-v_k))} \frac{D(\beta,t)^a}{D(\beta+\gamma v,t)^a}
    .
  \end{equation}
  With $e^t$ replaced by $e^{2at}$ but otherwise exactly as in the
  argument leading to \cite[(2)]{1907.07949}, by using that
    $s^{-1}$ and $q$ are H\"{o}lder conjugate and using the second
    part of~\eqref{e:Zmon},
  \begin{align}
    \E_{Q_{\beta,a}}(e^{2a s t_k})=\E_{Q^\gamma_{\beta,a}}\pa{{\frac{dQ_{\beta,a}}{dQ^\gamma_{\beta,a}}}e^{2a s t_k}}
  &  \leq \E_{Q^\gamma_{\beta,a}}\pa{\pa{\frac{dQ_{\beta,a}}{dQ^\gamma_{\beta,a}}}^q}^{1/q} \pa{\E_{Q^\gamma_{\beta,a}}\p{e^{2a t_k}}}^s
   \nnb
    &\leq \E_{Q^\gamma_{\beta,a}}\pa{\pa{\frac{dQ_{\beta,a}}{dQ^\gamma_{\beta,a}}}^q}^{1/q} e^{-2a s\gamma}.
  \end{align}
  The expectation on the right-hand side is estimated as in
  \cite{1907.07949}, with the only change that $\sqrt{D(\beta,t)}$ is
  replaced by $D(\beta,t)^a$ in all expressions, and that the change
  of measure from $Q_{\beta,a}$ to $Q_{\tilde \beta,a}$ involves the
  normalisation constants, i.e., a factor
  $Z_{\tilde \beta,a}/Z_{\beta,a}$. Setting
    $\gamma'= \gamma(q-1)$, we obtain
  \begin{align}
    \E_{Q^\gamma_{\beta,a}}\pa{\pa{\frac{dQ_{\beta,a}}{dQ^\gamma_{\beta,a}}}^q}
    &=
    \E_{Q^{\gamma'}_{\beta,a}}\pa{\pa{\frac{dQ_{\beta,a}}{dQ^\gamma_{\beta,a}}}^{q-1}\frac{dQ_{\beta,a}}{dQ^{\gamma'}_{\beta,a}}}
      \nnb
    &\leq
    \E_{Q^{\gamma'}_{\beta,a}}\pa{\frac{q}{2} \sum_{i,k} \beta_{ik}\cosh(t_i-t_k+\gamma'(v_i-v_k))(2q^2\gamma^2 (v_i-v_k)^2)}
      \nnb
    &=
      e^{\frac12 \sum_{i,k} \beta_{ik} { q^3}\gamma^2 (v_i-v_k)^2}
      \frac{Z_{\tilde \beta,a}}{Z_{\beta,a}}\E_{Q_{\tilde \beta,a}}\pa{\pa{\frac{D(\beta,t)}{D(\tilde\beta,t)}}^a}
  \end{align}
  where
  \begin{equation}
    \tilde \beta_{ik} = \beta_{ik}(1-2q^3\gamma^2(v_i-v_k)^2) \in
    [\frac12 \beta_{ik}, \beta_{ik}].
  \end{equation}
  The ratio of determinants is bounded using the matrix-tree theorem
  as done on \cite[p.7]{1907.07949},
  and we use that $Z_{\tilde \beta,a} \leq Z_{\beta,a}$, by \eqref{e:Zmon}.
  The result is \eqref{e:Sabotbd}.
\end{proof}

\begin{proof}[Proof of Theorem~\ref{thm:Z2}] 
  We may choose $s=1/(2a) = 1/3  \in (0,1)$ in Lemma~\ref{lem:Sabot}.
  We then combine \eqref{e:connspin-horo-bis} and \eqref{e:Sabotbd}
  and choose $v$ as a difference of Green functions (exactly as in \cite[Section~2.2]{1907.07949})
  to find that,
  \begin{equation}
    \P_\beta[0\leftrightarrow j] =\E_{Q_{\beta,a}}(e^{t_j}) =\E_{Q_{\beta,a}}(e^{2ast_j}) \leq |j|^{-c_\beta}
  \end{equation}
  as needed.
\end{proof}

\subsection{Mermin--Wagner theorem}

We now show that the vanishing of the density of the cluster
containing a fixed vertex on the torus also follows from a version of
the classical Mermin--Wagner theorem.  We first derive an expression
for a quantity closely related to the mean tree size.  For constant
$h$, Theorem~\ref{thm:H02forest} implies that
\begin{equation}
  \cb{z_{a}}_{\beta,h} = \sum_{F\in\cc F}\prod_{ij\in F}\beta_{ij}
  \prod_{T\in F}(1+\sum_{k\in T}(h-1_{a=k})),
\end{equation}
which leads to
\begin{align}
  \label{eq:MW2}
  \ab{z_{i}}_{\beta,h} &= 
                         \E_{\beta,h} \frac{h\abs{T_{i}}}{1+h\abs{T_{i}}},
\end{align}
where $T_i$ is the (random) tree containing the vertex $i$.

Let $\Lambda$ be a $d$-dimensional discrete torus, and let $\lambda(p)$ by the
Fourier multiplier of the corresponding discrete Laplacian:
\begin{equation}
  \label{e:lambda}
  \lambda(p) \bydef \sum_{j\in \Lambda} \beta_{0j} (1-\cos(p\cdot j)),
  \qquad p\in \Lambda^{\star}
\end{equation}
where $\cdot$ is the Euclidean inner product on $\R^d$ and
$\Lambda^{\star}$ is the Fourier dual of the discrete torus $\Lambda$.

\begin{theorem} \label{thm:MW}
  Let $d \geq 1$, and let $\Lambda$ be a $d$-dimensional discrete torus of side length $L$. Then
  \begin{equation} \label{e:MW1}
    \frac{1}{\ab{z_0}_{\beta,h}} \ge 1+ \frac{1}{(2\pi L)^d} \sum_{p \in \Lambda^{\star}} \frac{1}{\lambda(p) + h}.
  \end{equation}
\end{theorem}

\begin{proof}
  The proof is analogous to \cite[Theorem~1.5]{MR4021254}.
  We write the $\HH^{0|2}$ expectations $\ab{\xi_i\eta_j}_{\beta,h}$ and $\ab{z_i}_{\beta,h}$
  in horospherical coordinates using Corollary~\ref{cor:et}:
  \begin{equation} \label{e:MW-z0-pf}
    \avg{\xi_i\eta_j}_{\beta,h}= \avg{s_is_je^{t_i+t_j}}_{\beta,h},
    \quad
    \avg{z_i}_{\beta,h} = \avg{e^{t_i}}_{\beta,h} = \avg{e^{2t_i}}_{\beta,h}.
  \end{equation}
  Set
  \begin{equation}
    S(p) = \frac{1}{\sqrt{|\Lambda|}} \sum_j e^{i(p\cdot j)} e^{t_j}s_j,
    \quad
    D = \frac{1}{\sqrt{|\Lambda|}} \sum_j e^{-i(p\cdot j)} \ddp{}{s_j}.
  \end{equation}
  Since the expectation of functions depending only on $(s,t)$ in
  horospherical coordinates is an expectation with respect to a
  probability measure, denoted $\ab{\cdot}$ from hereon, the
  Cauchy--Schwarz inequality implies
  \begin{equation} \label{e:CS}
    \avg{|S(p)|^2} \geq \frac{|\avg{S(p)D\Hhoro}|^2}{\avg{|D\Hhoro|^2}}.
  \end{equation}
  Since the density in horospherical coordinates is $e^{-\Hhoro(s,t)}$,
  the probability measure $\avg{\cdot}$ obeys the integration by parts
  $\avg{FD\tilde H} = \avg{DF}$ identity for any function $F=F(s,t)$
  that does not grow too fast. Therefore by translation invariance,
  with $y_i = s_ie^{t_i}$,
  \begin{align}
    \label{e:S2-susy}
    \avg{|S(p)|^2}
    &= \frac{1}{|\Lambda|} \sum_{j,l} e^{i p\cdot(j-l)} \avg{y_jy_l}
      = \frac{1}{|\Lambda|} \sum_{j,l} e^{i p\cdot(j-l)} \avg{y_0y_{j-l}}
      = \sum_{j} e^{i (p\cdot j)} \avg{y_0y_{j}},
    \\
    \label{e:SH-susy}
    \avg{S(p)D\Hhoro}
    &= \avg{DS(p)}
      = \frac{1}{|\Lambda|} \sum_{j,l} e^{ip\cdot(j-l)}\avg{\ddp{y_j}{s_l}}
      = \frac{1}{|\Lambda|} \sum_{j} \avg{e^{t_j}}
      = \avg{z_0}.
  \end{align}
  By Cauchy--Schwarz, translation invariance, and \eqref{e:MW-z0-pf} we also have 
  \begin{equation}
    \label{e:af}
    \avg{e^{t_j+t_l}}
    \leq \avg{e^{2t_0}} = \avg{z_0}.
  \end{equation}
  Using \eqref{e:af} and the integration by parts identity it
  follows that 
  \begin{equation}
    \avg{|D\Hhoro|^2}
    =
    \avg{D\bar D\Hhoro}
    =
      \frac{1}{|\Lambda|} \sum_{j,l} \beta_{jl} \avg{e^{t_j+t_l}} (1-\cos(p\cdot (j-l))) + \frac{h}{|\Lambda|} \sum_j \avg{e^{t_j}}
    \leq
    \avg{z_0} ( \lambda(p) + h ).
  \end{equation}
  In summary, we have proved
  \begin{equation}
    \sum_{j} e^{i (p\cdot j)} \avg{\xi_0\eta_{j}}
    =
    \sum_{j} e^{i (p\cdot j)} \avg{y_0y_{j}}
    =
    \avg{|S(p)|^2}
    \geq \frac{|\avg{S(p)D\Hhoro}|^2}{\avg{|D\Hhoro|^2}}
    \geq
    \frac{\avg{z_0}}{\lambda(p) + h}
  \end{equation}
  Summing over $p \in \Lambda^{\star}$ in the Fourier dual of
  $\Lambda$ (with the sum correctly normalized), the left-hand side
  becomes $\avg{\xi_0\eta_0}$. Using $\avg{z_0} = 1-\avg{\xi_0\eta_0}$
  this then gives the claim:
  \begin{equation}
    \frac{1}{\avg{z_0}}-1 \ge \frac{1}{(2\pi L)^d} \sum_{p\in\Lambda^*} \frac{1}{\lambda(p) + h}. \qedhere
  \end{equation}
\end{proof}

From the Mermin--Wagner theorem we obtain that on a finite torus of side length $L$
the density of the tree containing $0$ tends to $0$ as
$L\to\infty$. We write $\lesssim$ for inequalities that hold up
  to universal constants.

\begin{corollary}
Let $\Lambda$ be the $2$-dimensional discrete torus of side length $L$. Then
\begin{equation}
  \E_{\beta,0} \frac{|T_0|}{|\Lambda|}  \lesssim \frac{1}{\sqrt{\log L}}.
\end{equation}
\end{corollary}

\begin{proof}
For any $h \leq 1/|\Lambda|$ we have $h|T_0| \leq 1$.
By Theorem~\ref{thm:MW}, for $d=2$ thus
\begin{equation}
  \E_{\beta,h} \frac{|T_0|}{|\Lambda|}
  =
  \frac{1}{|\Lambda|h} \E_{\beta,h} h|T_0|
  \leq
  \frac{2}{|\Lambda|h} \E_{\beta,h} \frac{h|T_0|}{1+h|T_0|}
  =
  \frac{2}{|\Lambda|h} \avg{z_0}_{\beta,h}
  \lesssim
  \frac{1}{h L^2 \log L}
\end{equation}
where we used 
that, for all $h \geq 0$, the Green's function of the discrete torus satisfies
\begin{equation} \label{e:gfbd}
  \frac{1}{(2\pi L)^2} \sum_{p \in \Lambda^{\star}} \frac{1}{\lambda(p) + h}
  \gtrsim \log (h^{-1}\wedge L).
\end{equation}
Directly following the conclusion of the present proof, we shall show that if $X$ is a random variable with $|X|\leq 1$,
and if $h \ll 1/|\Lambda|$,
\begin{equation} \label{e:EhE0bd}
  \absa{\E_{\beta,h} X -  \E_{\beta,0} X} = O(h|\Lambda|).
\end{equation}
Applying this estimate with $X=|T_0|/|\Lambda|$, for $h \ll 1/|\Lambda|$ we have
\begin{equation} \label{e:EhE0bd}
  \absa{
    \E_{\beta,h} \frac{|T_0|}{|\Lambda|} -
    \E_{\beta,0} \frac{|T_0|}{|\Lambda|}
  }
  = O(hL^2).
\end{equation}
With $h =L^{-2}(\log L)^{-1/2}$, combining both estimates gives
\begin{equation}
  \E_{\beta,0} \frac{|T_0|}{|\Lambda|}
  \lesssim \frac{1}{hL^2 \log L} + h L^2
  \lesssim \frac{1}{\sqrt{\log L}}  .
  \qedhere
\end{equation}
\end{proof}

\begin{lemma}
Let $\Lambda$ be any finite graph with $|\Lambda|$ vertices.
Let $X$ be a random variable with $|X|\leq 1$.
Then for $h \ll 1/|\Lambda|$,
\begin{equation} \label{e:EhE0bd}
  \absa{\E_{\beta,h} X -  \E_{\beta,0} X} = O(h|\Lambda|).
\end{equation}
\end{lemma}

\begin{proof}
  By definition,
  \begin{equation}
    \E_{\beta,h} X
    = \frac{\E_{\beta,0} (X \prod_{T\in F} (1+h|T|))}{\E_{\beta,0}(\prod_{T\in F} (1+h|T|))}.
  \end{equation}
  With $A'/(1+\epsilon)-A = (A'-A) - A' (\epsilon/(1+\epsilon)) = (A'-A) + (A'/(1+\epsilon)) \epsilon$ we get
  \begin{align}
    \E_{\beta,h}X - \E_{\beta,0} X
    = \E_{\beta,0}(X (\prod_{T} (1+h|T|)-1)) - \E_{\beta,h}(X)\E_{\beta,0}(\prod_{T} (1+h|T|)-1)  .
  \end{align}
  Since $|X|\leq 1$ it suffices to bound
  \begin{equation}
    \prod_{T\in  F} (1+h |T|)-1
    =
    \sum_{F' \subset F}\prod_{T\in  F'} h |T|
  \end{equation}
  where the sum runs over subforests $F'$ of $F$, i.e., unions of the disjoint trees in $F$.
  Since $\sum_i |T_i| \leq |\Lambda|$,
  \begin{equation}
    \sum_{F' \subset F}\prod_{T\in  F'} h |T|
    \leq \sum_{n\geq 1} \sum_{i_1, \dots, i_n} \prod_{i=1}^n (h|T_i|)
    \leq  \sum_{n\geq 1} \pa{h\sum_{i} |T_i|}^n
    \leq  \sum_{n\geq 1} (h|\Lambda|)^n = O(h|\Lambda|)
  \end{equation}
  whenever $h|\Lambda|\ll 1$.
\end{proof}

\appendix
\section{Percolation properties}
\label{sec:Perc}

In this appendix we indicate how to deduce Theorem~\ref{thm:Z2} from
our results in Section~\ref{sec:Z2}. We also give proofs
of the other unproven claims from Section~\ref{sec:model}. While we
are unaware of any references for these results, it is likely that
they have been independently discovered in the past.  In particular,
we thank G.\ Grimmett for pointing out Proposition~\ref{prop:dom}.

\subsection{Stochastic domination}
\label{sec:stoch-domin}

The proof of Proposition~\ref{prop:dom} is an application of Holley's
inequality, and we begin by recalling the set-up and result. For a
finite set $X$ and probability measures
$\mu_{i}\colon 2^{X}\to [0,\infty)$, $\mu_{1}$ \emph{convexly
  dominates} $\mu_{2}$ if for all $A,B\subset 2^{X}$
\begin{equation}
  \label{eq:CD}
  \mu_{1}(A\cup B)\mu_{2}(A\cap B) \geq \mu_{1}(A)\mu_{2}(B).
\end{equation}
Holley's inequality, as stated in~\cite{MR865241}, says that
$\mu_{1}$ convexly dominating $\mu_{2}$ is a sufficient condition
for $\mu_{1}$ to stochastically dominate $\mu_{2}$.  

\begin{proof}[Proof of Proposition~\ref{prop:dom}]
  To prove the proposition, we verify the condition
  \eqref{eq:CD} when $\mu_{1}$ is $p_{\beta}$ bond
  percolation and $\mu_{2}$ is the arboreal gas with parameter
  $\beta$. This is straightforward: if $B$ is not a forest the inequality is trivial because
  the right-hand side is $0$, whereas if $B$ is a forest then both
  sides are actually equal.
\end{proof}

\begin{remark}
  \label{rem:coupling}
  Proposition~\ref{prop:dom} implies a monotone coupling between the
  arboreal gas with parameter $\beta$ and $p_{\beta}$-bond
  percolation exists. An explicit construction of such a coupling would be
  interesting.
\end{remark}

\subsection{The arboreal gas in infinite volume}
\label{sec:arbor-gas-infin}

Let $\Lambda \subset \Z^{d}$ be a finite set of vertices such that the
subgraph $\bbG_{\Lambda}=(\Lambda, E(\Lambda))$ induced by $\Lambda$ is
connected. Write $\P_{\Lambda,\beta}$ for the arboreal graph measure
on $\bbG_{\Lambda}$. In this section we prove
Proposition~\ref{prop:subseq}, i.e., we show how
Conjecture~\ref{con:NA} implies the existence of the infinite-volume
limit $\lim_{\Lambda\uparrow \Z^{d}}\P_{\Lambda,\beta}$, where
$\Lambda_{n}\uparrow \Z^{d}$ means that $\Lambda_{n}$ is increasing
and for any finite set $A\subset \Z^{d}$, there is an $n_{A}$ such
that $A\subset \Lambda_{n}$ for $n\geq n_{A}$.

\begin{proof}[Proof of Proposition~\ref{prop:subseq}]
    We consider the case of general non-negative weights
    $\beta=(\beta_{ij})$. We first claim it suffices to prove that for any
    finite graph $\bbG=(V,E)$, any set $\tilde E$ of edges and any
    $e\notin \tilde E$, that
    \begin{equation}
      \label{eq:mt}
      \P_{\bbG,\beta}\q{\tilde E\cup \{e\}} \leq \P_{\bbG,\beta}\q{\tilde E}\P_{\bbG,\beta}\q{e}.
    \end{equation}
    Note that this implies $\P_{\bbG,\beta}\q{\tilde E}$ is (weakly)
    monotone decreasing in $\beta_{ij}$ for all edges
    $ij\notin \tilde E$. The sufficiency of this claim is a standard
    argument, but we provide it for completeness.
    
    Observe that monotonicity and probabilities being bounded below by
    zero implies that for any finite collection of edges $\tilde E$ in
    $\Z^{d}$, $\lim_{n\to\infty}\P_{\bbG_{n},\beta}\q{\tilde E}$
    exists. This is because the transition from $\bbG_{n}$ to $\bbG_{n+1}$
    can be viewed as a limit when $\beta^{(n)}_{ij}$ (weakly)
    increases to $\beta^{(n+1)}_{ij}$ -- the increase is in fact no
    change for $ij\in E(\bbG_{n})$ and is positive for
    $ij\notin E(\bbG_{n})$.  Moreover, the limit is independent of the
    sequence $\bbG_{n}$, as can be seen by interlacing any two sequence
    $\bbG^{(i)}_{n}$ that increase to $\Z^{d}$. By inclusion-exclusion
    the probability of any cylinder event depending on edges $\tilde E$
    can be expressed in terms of the occurrence of finite subsets of
    edges in $\tilde E$, and hence every cylinder event has a
    well-defined limiting probability. Since all cylinder
    probabilities converge, there is a well-defined probability measure
    $\P_{\beta}$ on $\{0,1\}^{E(\Z^{d})}$ that is the weak limit of
    the $\P_{\bbG_{n},\beta}$. Moreover, $\P_{\beta}$ is translation invariant by the
    interlacing argument used above.

    What remains is to prove~\eqref{eq:mt}. This is obvious if
    $\tilde E$ is the empty set of edges, so we may assume $\tilde E$
    is non-empty. We use an argument of Feder--Mihail~\cite{FM}.  In
    the proof of \cite[Lemma~3.2]{FM} it is shown that~\eqref{eq:mt}
    follows if one knows, for all finite graphs $\bbG=(V,E)$, that
    \begin{enumerate}
    \item $\P_{\bbG,\beta}\cb{e,f} \leq \P_{\bbG,\beta}\cb{e}\P_{\bbG,\beta}\cb{f}$
      for all distinct $e,f\in E$, and
    \item For any $\tilde E\subset E$ and $e\notin \tilde E$, there is
      an $f\in E$ such that $\P_{\bbG,\beta}\q{\tilde E\mid e,f} \geq
      \P_{\bbG,\beta}\q{\tilde E\mid e,\bar{f}}$, where $\bar{f}$ means $f$
      is not present.
    \end{enumerate}
    The first of these conditions is precisely
    Conjecture~\ref{con:NA}. The second is obvious: choose $f\in
    \tilde E$, for which the right-hand side is zero.
\end{proof}

\subsection{Proof of  Corollary~\ref{cor:Z2}}
\label{sec:Z2qual}
In this section we show how to deduce Corollary~\ref{cor:Z2} from the
quantitative estimate of
Theorem~\ref{thm:Z2}; 
we thank Tom Hutchcroft for suggesting this proof. The proof crucially
exploits planar duality and the resulting connected subgraph model
that is dual to the arboreal gas. The precise definitions are as
follows.

Given a set $\omega\in \{0,1\}^{E(\Z^{2})}$, we write $\omega^{\star}$
for the dual set of edges, i.e., if
$e^{\star}$ is the edge dual to $e$, then
$\omega^{\star}_{e^{\star}}\bydef 1-\omega_{e}$. In what
follows we will identify $\Z^{2}$ with its dual; with this
identification $\omega\mapsto \omega^{\star}$ is an involution on the
set of edge configurations $\{0,1\}^{E(\Z^{2})}$.

Suppose $\P_{\beta}$ is an arboreal gas measure, either on a
finite graph, or a weak limit of measures on finite graphs.
We define the \emph{connected subgraph measure} $\P_{\beta}^{\star}$ by
$\P_{\beta}^{\star}(A^{\star}) = \P_{\beta}(A)$ for all edge
configurations $A$. The name arises as for finite-volume measures
$\P_{\beta}^{\star}$ is supported on connected subgraphs of $\Z^{2}$
since $\P_{\beta}$ is supported on forests with finite components,
see, e.g., \cite[Theorem~2.1]{MR2060630}. It is important to note,
however, that this is not necessarily true for infinite-volume
measures: in this case it may be that $\P_{\beta}^{\star}$ has
disconnected graphs in its support.

\begin{remark}
  \label{rem:CSM}
  The connected subgraph measure as defined above is a special case of
  a more general construction that occurs in the context of $q\to 0$
  limits of the $q$-state random cluster model,
  see~\cite{MR2060630}.
\end{remark}

Given an event $A\subset \{0,1\}^{E(\Z^{2})}$, we write $A_{e} = \{
\omega\cup \{e\} \mid \omega \in A\}$ and $A^{e} = \{ \omega\setminus
\{e\} \mid \omega \in A\}$ for the events in which we add or remove
the edge $e$, respectively. 
\begin{lemma}
  \label{lem:insert-tolerant}
  For any arboreal gas measure $\P_{\beta}$, the dual measure
  $\P_{\beta}^{\star}$ is \emph{insertion tolerant}, i.e., for $A\subset
  \{0,1\}^{E(\Z^{2})}$ and any edge $e$,
  \begin{equation}
    \label{eq:IT}
    \P_{\beta}^{\star}\cb{A_{e}} > 0 \qquad \text{if  $\P_{\beta}^{\star}\cb{A}>0$}.
  \end{equation}
\end{lemma}
\begin{proof}
  This is equivalent to proving that the arboreal gas is deletion
  tolerant, i.e., that $\P_{\beta}\cb{A^{e}}>0$ if
  $\P_{\beta}\cb{A}>0$. We will need a standard notion of boundary
  conditions~\cite[Section~1.2.1]{MR4043224}. In brief, for a
  finite-volume $\Lambda$, a \emph{boundary condition} $\omega$ is a
  partition of the boundary vertices of $\Lambda$. Configurations are
  valid for a given boundary condition if they are forests after
  identifying each set of the partition together.
  For any finite-volume $\Lambda$, any boundary condition $\omega$,
  and any forest $F$,
  \begin{equation*}
    \P_{\Lambda,\beta}^{\omega}\cb{F^{e}}\geq \text{min}(1/\beta,1)  
    \P_{\Lambda,\beta}^{\omega}\cb{F},
  \end{equation*}
  and hence the same inequality holds true for all events. Following a
  standard argument (e.g.,~\cite[Theorem~4.17 (b)]{MR2243761}) implies
  this inequality transfers to the infinite volume limit.
\end{proof}

Recall that a \emph{ray} is a semi-infinite self-avoiding walk. Two
rays $\gamma_{1}$ and $\gamma_{2}$ are \emph{equivalent} if there is
no finite set of vertices $X$ that separates infinitely many vertices
of $\gamma_{1}$ from infinitely many vertices of $\gamma_{2}$. This is
an equivalence relation, and equivalence classes are called
\emph{ends}. 

\begin{proposition}
  \label{prop:BK-CSM}
  For any translation invariant connected subgraph measure
  $\P_{\beta}^{\star}$ on $\Z^{2}$, the number of components is
  $\P_{\beta}^{\star}$-a.s.\ one. Further, the number of ends of the
  random subgraph with law $\P_{\beta}^{\star}$ is almost surely in
  $\{1,2\}$.
\end{proposition}
\begin{proof}
  Since translations act transitively on $\Z^{2}$,
  \cite[Theorem~7.9]{MR3616205} implies that there is at most one
  infinite component under $\P_{\beta}^{\star}$. To complete the proof
  of the first conclusion, note that for any fixed $K\in\ N$, for all
  sufficiently large volumes the finite-volume connected subgraph
  measures give probability zero to the existence of a cluster of size
  at most $K$.

  The second claim is well known, see, e.g., \cite[Exercise~7.24]{MR3616205}.
\end{proof}

\begin{lemma}
  \label{lem:topo}
  For any infinite-volume translation invariant arboreal gas measure
  $\P_{\beta}$ there are at most two infinite trees.
\end{lemma}
\begin{proof}
  Note first that translation invariance of $\P_{\beta}$ implies translation invariance of $\P_{\beta}^*$.  Next, we note that almost surely all infinite trees in the arboreal
  gas are one-ended: if not, there is a positive probability of the
  arboreal gas containing a bi-infinite path. The dual of this
  bi-infinite path is an edge cut of $\Z^{2}$, contradicting the
  almost sure connectedness of the dual of the arboreal gas from
  Proposition~\ref{prop:BK-CSM}.

  If the arboreal gas contains three infinite trees with positive
  probability, then there exist three disjoint semi-infinite paths
  $\gamma_{i}$ with initial vertex $x_{i}$, $i=1,2,3$. Fix a ball $B$
  containing the $x_{i}$, and note that the dual of the edges in
  $B\cup \bigcup_{i=1}^{3}\gamma_{i}$ divides $\Z^{2}$ into three
  connected components. Since the dual to the arboreal gas is
  connected, it contains an infinite path in each of these components,
  which implies it has at least three ends.  By
  Proposition~\ref{prop:BK-CSM} this is a contradiction.
\end{proof}

\begin{proof}[Proof of Corollary~\ref{cor:Z2}] 
  Let $T_{0}$ denote the tree containing the origin. By translation
  invariance and ergodic decomposition,
  $\P_{\beta}\cb{\text{$T_{0}$ is infinite}}$ is the density of
  the vertices in infinite trees. Moreover, by an adaptation
  of~\cite[Theorem~1]{MR990777}, each individual infinite tree has a
  well-defined density. We now argue by contradiction. Suppose that
  $\P_{\beta}\cb{\text{$0$ is in an infinite tree}}=p>0$. By
  Lemma~\ref{lem:topo}, this implies the existence of an infinite tree
  with a positive density, and hence of a $p'>0$ such that
  \begin{equation}
    \P_{\beta}\cb{\text{$T_{0}$ has positive density}}=p'.
  \end{equation}
  This is a contradiction, as Theorem~\ref{thm:Z2} implies that the
  expected density of $T_{0}$ is zero in any infinite-volume limit.
\end{proof}

\section{Rooted spanning forests and the uniform spanning tree}
\label{app:UST}

For the reader's convenience, we include a short summary of the
well-known representation of \emph{rooted} spanning forests and
uniform spanning \emph{trees} in the terms of the fermionic Gaussian
free field (fGFF).  We follow the notation of Section~\ref{sec:repr}.
The fGFF is the unnormalised expectation on
$\Omega^{2\Lambda}$ defined by
\begin{equation}
  [F]_{\beta,h}^{\rm fGFF} \bydef \ob{\prod_{i\in\Lambda} \partial_{\eta_i}\partial_{\xi_i}}
  \exp\qB{(\V{\xi},\Delta_\beta \V{\eta}) + (\V{h},\V{\xi\eta})} F.
\end{equation}
where $\V{\xi\eta} \bydef (\xi_i\eta_i)_i$.
The normalised version is again denoted by $\avg{\cdot}^{\rm fGFF}_{\beta,h}$ if
$[1]^{\rm fGFF}_{\beta,h}>0$; see Section~\ref{sec:repr}.
It is straightforward that
the fGFF is the properly normalised $\beta\to\infty$ limit of the $\HH^{0|2}$ model
as stated in the following fact;
we omit the details. 
\begin{fact}
  For all weights $\V{\beta}$ and $\V{h}$,
  \begin{equation}
    [F(\V{\xi}, \V{\eta})]_{\beta,h}^{\rm fGFF}
    =
    \lim_{\alpha \to \infty}
    \frac{1}{\alpha^{|\Lambda|}}\qa{F(\sqrt{\alpha}\V{\xi},\sqrt{\alpha}\V{\eta})}_{\alpha \beta,\alpha h}
    ,
  \end{equation}
  where the unnormalised expectation on the right-hand side is that of the $\HH^{0|2}$ model.
\end{fact}

As a consequence of this fact and Theorem~\ref{thm:H02forest}, the
partition function of the fGFF can be expressed in terms of weighted
\emph{rooted} spanning forests.  Let $\mathcal{F}_{\rm root}$ denote
the set of all spanning forests together with a choice of root vertex
in each tree of the forest.
\begin{corollary}
  \label{cor:rf}
  For all weights $\V{\beta}$ and $\V{h}$,
  \begin{equation}
    [1]^{\rm fGFF}_{\beta,h} = \sum_{F \in \mathcal{F}_{\rm root}} \prod_{(T,r) \in F} \ob{\prod_{ij \in T} \beta_{ij}} h_r
    .
  \end{equation}
\end{corollary}
Corollary~\ref{cor:rf} also has an elementary proof: it can be seen as a
consequence of the matrix-tree theorem.

The case of the uniform spanning tree (UST) is obtained by pinning the
fGFF at a single arbitrary
vertex which we denote $0$. This corresponds to taking $h_j = 1_{j=0}$,
or equivalently to adding a factor $\xi_0\eta_0$ inside the expectation.
In analogy to Section~\ref{sec:repr}, we denote the pinned expectation by
an additional superscript $0$, i.e.,
\begin{equation}
  [F]_\beta^{{\rm fGFF},0}  = [\xi_0\eta_0 F]_\beta^{\rm fGFF}.
\end{equation}
The following corollary is then immediate from the previous one.

\begin{corollary}
  For all sets of edges $S$,
  \begin{equation}
    \P_\beta^{\rm UST}[S] = \cb{\prod_{ij\in S}\beta_{ij} (\xi_i-\xi_j)(\eta_i-\eta_j)}_{\beta}^{{\rm fGFF},0}.
  \end{equation}
\end{corollary}

For the UST, it is well-known that negative association holds,
i.e., that the occurrence of disjoint edges $ij,kl$ are
negatively correlated.
Various proofs exist, see e.g.~\cite{MR2060630,FM}.
We include a new proof that mimics the proof of the Ginibre
inequality~\cite{MR0269252}.

\begin{proposition}
  For the uniform spanning tree, negative association holds:
  for all distinct $ij$ and $kl$,
  \begin{equation}
    \P_\beta^{\rm UST}[ij,kl] \leq \P_\beta^{\rm UST}[ij]\P_\beta^{\rm UST}[kl].
  \end{equation}
\end{proposition}
\begin{proof}
  Consider the doubled Grassman algebra $\Omega^{4\Lambda}$ with generators $\xi_i,\eta_i,\xi'_i,\eta'_i$ where $i \in \Lambda'$. Abusing notation, we write $\avg{\cdot}$
  for the product of the two fGFF expectations, i.e.,
  \begin{equation}
    \avg{F(\xi,\eta)G(\xi',\eta')}
    =    \avg{F(\xi,\eta)}^{\rm fGFF,0}\avg{G(\xi,\eta)}^{\rm fGFF,0}.
  \end{equation}
  Set $\chi_{ij} = (\xi_i-\xi_j)(\eta_i-\eta_j)$ and define $\chi_{ij}'$ analogously. Then
  \begin{equation}
    \label{eq:grhs}
    \P_\beta^{\rm UST}[ij,kl] - \P_\beta^{\rm UST}[ij]\P_\beta^{\rm UST}[kl]
    = \frac12 \beta_{ij}\beta_{kl} \avg{(\chi_{ij}-\chi_{ij}')(\chi_{kl}-\chi_{kl}')}.
  \end{equation}
  Mimicking Ginibre \cite{MR0269252}, we change generators in $\Omega^{4\Lambda}$
  according to
  \begin{align}
    \xi_i \mapsto \frac{1}{\sqrt{2}}(\xi_i+\xi_i'),\quad
    \eta_i \mapsto \frac{1}{\sqrt{2}}(\eta_i+\eta_i'),\quad
    \xi_i' \mapsto \frac{1}{\sqrt{2}}(\xi_i-\xi_i'),\quad
    \eta_i' \mapsto \frac{1}{\sqrt{2}}(\eta_i-\eta_i').
  \end{align}
  The action defining the product of two fGFFs is invariant
  under this change of generator and the integrand of the RHS
  of~\eqref{eq:grhs} transforms as
  \begin{alignat}{1}
    (\chi_{ij}-\chi_{ij}')(\chi_{kl}-\chi_{kl}')
    \mapsto
    & - (\xi_i-\xi_j)(\xi_k-\xi_l)(\eta_i'-\eta_j')(\eta_k'-\eta_l')
      \nnb
    & - (\eta_i-\eta_j)(\eta_k-\eta_l)(\xi_i'-\xi_j')(\xi_k'-\xi_l')
    \nnb
    & - (\xi_i-\xi_j)(\eta_k-\eta_l)(\xi_k'-\xi_l')(\eta_i'-\eta_j')
    \nnb
    & - (\xi_k-\xi_l)(\eta_i-\eta_j)(\xi_i'-\xi_j')(\eta_k'-\eta_l').
  \end{alignat}
  Taking the expectation, only the last two terms contribute since only monomials
  with the same number of factors of $\xi$ as $\eta$ have non-vanishing expectation, e.g., $\avg{\xi_i\xi_j}^{\rm fGFF}=0$. These last two terms give the same expectation:
  \begin{equation} \label{e:UST-deficit}
    \P_\beta^{\rm UST}[ij,kl] - \P_\beta^{\rm UST}[ij]\P_\beta^{\rm UST}[kl]
    = - \beta_{ij}\beta_{kl}\avg{(\xi_i-\xi_j)(\eta_k-\eta_l)}^{\rm fGFF,0}\avg{(\xi_k-\xi_l)(\eta_i-\eta_j)}^{\rm fGFF,0}.
  \end{equation}
  By \eqref{e:xieta-symmetry}
  the two terms in the product on the right-hand side are equal, and 
  hence the right-hand side is non-positive.
\end{proof}

\begin{remark}
  The right-hand side in \eqref{e:UST-deficit} gives an
    alternate expression for the deficit $\Delta_{ij,kl}^{2}$ that
    occurs in \cite[Theorem~2.1]{FM}.
\end{remark}

\section*{Acknowledgements}

We thank G.\ Grimmett for helpful discussions and for pointing out
Proposition~\ref{prop:dom}, G.\ Slade for helpful references
concerning Laplace's method, and B.\ T\'{o}th and A.\ Celsus for
helpful conversations.  We thank the anonymous referee for their
helpful comments and criticisms, and we especially thank T.\
Hutchcroft for pointing out an error in an earlier version of the
appendix, and for proposing the proof used in
Appendix~\ref{sec:Z2qual}.  T.H.\ was supported by EPSRC grant no.\
EP/P003656/1 and was at the University of Bristol when this work was
carried out. A.S.\ is supported by EPSRC grant no.\ 1648831.  N.C.\
{supported by Israel Science Foundation grant number 1692/17}.

\bibliographystyle{plain}
\bibliography{all}

\end{document}